\documentclass[psamsfonts]{amsart}  
\usepackage{amssymb}

\usepackage[T1]{fontenc}
\usepackage[T1]{fontenc}
\usepackage[utf8]{inputenc}
\usepackage{bera}
\usepackage[portrait,margin=3cm]{geometry} 
\usepackage[mathscr]{euscript}
\usepackage[colorlinks=true,linkcolor=blue,citecolor=blue,urlcolor=blue]{hyperref} 
\usepackage{amsmath,amssymb,amsthm,amsfonts,amsbsy,latexsym,dsfont,color,graphicx,enumitem}
\usepackage[numeric,initials,nobysame]{amsrefs}

\usepackage{amssymb,bbm}

\usepackage{enumitem}
\usepackage{color,comment}

\newcommand{\R}{\mathbb{R}}
\newcommand{\C}{\mathbb{C}}
\newcommand{\E}{\mathbb{E}}
\renewcommand{\P}{\mathbb{P}}

\newcommand{\cE}{\mathcal{E}}
\newcommand{\cF}{\mathcal{F}}

\newcommand{\Cue}{C_{\mathrm{UE}}}
\newcommand{\cue}{c_{\mathrm{UE}}}
\newcommand{\Cle}{C_{\mathrm{LE}}}

\newcommand{\cle}{c_{\mathrm{LE}}}

\newcommand{\UE}{\hyperref[def-UE]{\mathbf{UE}_{\beta}}}

\newcommand{\UEF}{\hyperref[def-UE]{\mathbf{UE}_{F}}}

\newcommand{\ULE}{\hyperref[def-ULE]{\mathbf{ULE}_{\beta}}}

\newcommand{\VD}{\hyperref[def-VD-alpha]{\mathbf{VD}_\alpha}}
\newcommand{\RVD}{\hyperref[def-RVD]{\mathbf{RVD}_{\alpha'}}}
\newcommand{\lamb}{\hyperref[def:lambda-beta]{\mathbf{\lambda_\beta}}}
\newcommand{\lambge}{\hyperref[def:lambda-beta]{\mathbf{\lambda_{\beta\ge}}}}
\newcommand{\lamble}{\hyperref[def:lambda-beta]{\mathbf{\lambda_{\beta\le}}}}
\newcommand{\lambF}{\hyperref[def:lambda-beta]{\lambda_{F}}}
\newcommand{\lambFge}{\hyperref[def:lambda-beta]{\lambda_{F\ge}}}
\newcommand{\lambFle}{\hyperref[def:lambda-beta]{\lambda_{F\le}}}

\newcommand{\FK}{\hyperref[df:FK]{\mathbf{FK}_{\beta}}}
\newcommand{\FKF}{\hyperref[df:FK]{\mathbf{FK}_{F}}}
\newcommand{\Eb}{\hyperref[def-E-beta]{\mathbf{E_\beta}}}
\newcommand{\Ebge}{\hyperref[def-E-beta]{\mathbf{E_\beta}_{\ge}}}
\newcommand{\Eble}{\hyperref[def-E-beta]{\mathbf{E_\beta}_{\le}}}
\newcommand{\EO}{\hyperref[def-EO]{\bar{\mathbf{E}}_\beta}}
\newcommand{\EOF}{\hyperref[def-EO]{\bar{\mathbf{E}}_{F}}}
\newcommand{\EF}{\hyperref[def-E-beta]{\mathbf{E}_{F}}}
\newcommand{\EFge}{\hyperref[def-E-beta]{\mathbf{E}_{F\ge}}}
\newcommand{\EFle}{\hyperref[def-E-beta]{\mathbf{E}_{F\le}}}

\newcommand{\Ebar}{\hyperref[def:e-bar]{\bar{\mathbf{E}}}}
\newcommand{\Ebarp}{\hyperref[df:e-bar-prime]{\bar{\mathbf{E}}^\prime}}

\newcommand{\Heis}{\mathbb{H}^{2n+1}}
\newcommand{\ch}{\mathcal{H}}
\newcommand{\BH}{\mathcal{B}}
\newcommand{\pa}{{    \gamma}}
\newcommand{\hs}{{   \varrho}}

\newtheorem{thm}{Theorem}[section]
\newtheorem*{thm*}{Theorem}
\newtheorem{corollary}[thm]{Corollary}

\newtheorem{lemma}[thm]{Lemma}

\newtheorem{remark}[thm]{Remark}

\newtheorem{prop}[thm]{Proposition}

\newtheorem{theorem}[thm]{Theorem}
\theoremstyle{definition}
\newtheorem{df}[thm]{Definition}
\newtheorem*{claim*}{Claim}

\newtheorem{example}{Example}[section]

\numberwithin{equation}{section}
\newtheorem*{acknowledgement}{Acknowledgement}

\DeclareMathOperator*{\esup}{ess\,sup}
\DeclareMathOperator*{\einf}{ess\, inf}

\begin{document}

\title[]{Spectral bounds for exit times on metric measure Dirichlet spaces and applications}

\author[Mariano]{Phanuel Mariano{$^\dag$}}
\thanks{\footnotemark {$\dag$} Research was supported in part by NSF Grant DMS-2316968
and an AMS-Simons Travel Grant 2019-2023.}
	\address{Department of Mathematics\\
		Union College\\
		Schenectady, NY 12308,  U.S.A.}
	\email{marianop@union.edu}

\author[Wang]{Jing Wang{$^{\ddag}$}}
\thanks{\footnotemark {$\ddag$} Research was supported in part by NSF Grant DMS-1855523 and NSF DMS-2246817}
\address{ Department of Mathematics\\
Purdue University\\
West Lafayette, IN 47907,  U.S.A.}
\email{jingwang@purdue.edu}



\begin{abstract}

Assuming the heat kernel on a doubling Dirichlet metric measure space has a sub-Gaussian bound, we prove an asymptotically sharp spectral upper bound on the survival probability of the associated diffusion process. As a consequence, we can show that the supremum of the mean exit time over all starting points is finite if and only if the bottom of the spectrum is positive. Among several applications, we show that the spectral upper bound on the survival probability implies a bound for the Hot Spots constant for Riemannian manifolds. Our results apply to interesting geometric settings including sub-Riemannian manifolds and fractals.


\end{abstract}

\maketitle

\tableofcontents

\section{Introduction}

Let $B_t,t\ge0$ be a Brownian motion in $\R^d$ started from $x\in \R^d$. For any domain $D\subset \R^d$ that contains $x$, we define the \emph{first exit time} from a domain $D$ by 
\[
\tau_{D}=\inf\left\{ t>0\mid B_{t}\notin D\right\}.
\]
Exit times of Brownian motion from Euclidean domains have been widely studied in the literature. Its distribution function $\mathbb{P}_{x}\left(\tau_{D}>t\right)$, referred to as the \emph{survival probability}, has strong connections with the spectral information of the domain. For instance, it is widely known that the large time asymptotic behavior of $\mathbb{P}_{x}\left(\tau_{D}>t\right)$ reflects  the bottom of spectrum of the domain $D$ since
\begin{equation}
\lim_{t\to\infty}\frac{1}{t}\log\mathbb{P}_{x}\left(\tau_{D}>t\right)=-\lambda\left(D\right),\label{eq:Intro-2}
\end{equation}
where $\lambda(D)$ is the \emph{bottom of the  Dirichlet spectrum} for $-\frac{1}{2}\Delta_D$ acting in $L^2(D)$   {given by 
\[
\lambda\left(D\right)=\inf_{f\in H_{0}^{1}\left(D\right),f\neq 0}\frac{\frac{1}{2}\int_{D}\left|\nabla f\right|^{2}dx}{\int_{D}f^{2}dx}.
\]
}
Another interesting interaction is due to MacDonald-Meyers \cite{McDonald-Meyers}, who  showed that the sequence of $L^1$ norm of exit moments given by   $\{ \int_D\mathbb{E}_x(\tau_D^n) d\mu\}_{n\ge1}$ determines the spectral sequence of Dirichlet eigenvalues $\{\lambda_n(D)\}_{n\ge1}$. We also point to the work of \cite{Colladay-Langford-McDonald-2018,Dryden-Langford-McDonald-2017} who proved various inequalities between the exit times moments and the Dirichlet eigenvalues on Riemannian manifolds. 

The $L^\infty$ norm of the mean exit time, namely  $\sup_{x\in D}\mathbb{E}_{x}\left[\tau_{D}\right]$, has also been extensively studied through its ``tug of war" against the bottom of the spectrum. In particular, for any domain $D\subset \mathbb{R}^d$, it is known that
\begin{equation}\label{Torsion}
1\leq \lambda(D) \sup_{x\in D}\mathbb{E}_{x}\left[\tau_{D}\right]\leq {M_d},
\end{equation}
so long as $\lambda(D)>0$ or $\sup_{x\in D}\mathbb{E}_{x}\left[\tau_{D}\right]<\infty$, where $M_d$ is a uniform constant that only depends on the dimension $d$. This was independently obtained through various levels of generality in the works of \cite{Banuelos-Carrol1994,Vandenberg-Carroll-2009,Giorgi-Smits-2010}.
 
Recently in \cite{Banuelos-Mariano-Wang-2020}, Ba\~nuelos and the present authors showed that there exists an asymptotically sharp constant $M_{p,d}$ depending only on  $p$ and $d$ such that for all $D\subset\mathbb{R}^{d}$ with $\lambda\left(D\right)>0$ and for any $p\geq1$, 
\begin{equation}
\Gamma\left(p+1\right)\leq\lambda\left(D\right)^{p}\cdot\sup_{x\in D}\mathbb{E}_{x}\left[\tau_{D}^{p}\right]\leq M_{p,d}. \label{eq:Intro-4}
\end{equation}
In regards to other papers most closely related to our work we also note the works of \cite{Vandenberg-2017a,Huang-etal-2022,Huang-Wang-2023}. In \cite{Vandenberg-2017a}, various spectral bounds have been obtained for Riemannian manifolds. More recently, related spectral bounds for exit times have also been obtained in the general metric measure setting in \cite{Huang-etal-2022,Huang-Wang-2023}. We also refer the reader to \cite{Boudabra-Markowsky2020b,Boudabra-Markowsky-2020,Mariano-Panzo-2020,Panzo2021,Vogt-2019a} for other similar results regarding $\lambda_{1}\left(D\right),\lambda\left(D\right)$, $\tau_D$ and exit time moments in various settings. 

The \textbf{goal} of this paper is to obtain results relating exit times and the bottom of the spectrum for general geometric settings, that include large classes of doubling metric measure Dirichlet spaces of recent interest, such as Carnot groups, sub-Riemannian manifolds with curvature lower bounds, and fractals. We seek for minimal assumptions on the diffusion processes to be considered on these metric measure spaces that are sufficient for observing similar phenomenon such as Brownian motions on $\R^d$.  Our work covers general domains that do not necessarily have a discrete spectrum, hence our methods requires a precise analysis of the heat kernel. As a consequence, our result provides several applications that are useful to other interesting problems in analysis and probability. This includes spectral bounds for the mean exit times, a lower bound and large time asymptotic for the survival probability, an estimate for the Hot Spot constant for Riemannian manifolds, and a characterization for heat kernel upper bounds.





Let $\left(M,d,\mu\right)$ be a metric measure space endowed with a regular, local, Dirichlet form $\left(\mathcal{E},\mathcal{F}\right)$, which admits a heat kernel $p_{M}(x,y,t)$ for $(x,y,t)\in M\times M \times (0,\infty)$. {Recall that a \emph{Dirichlet form} $\left(\mathcal{E},\mathcal{F}\right)$ in
$L^{2}\left(M,\mu\right)$ is a bilinear form $\mathcal{E}:\mathcal{F}\times\mathcal{F}\to\mathbb{R}$
defined on a dense subspace $\mathcal{F}$ of $L^{2}\left(M,\mu\right)$
that satisfies the following three properties: (1) Positivity: $\mathcal{E}\left(f\right):=\mathcal{E}\left(f,f\right)\geq0$,
(2) Closedness: $\mathcal{F}$ is a Hilbert space with respect to
the inner product given by $\mathcal{E}_{1}\left(f,g\right):=\mathcal{E}\left(f,g\right)+\left(f,g\right)$,
and (3) Markov Property: if $f\in\mathcal{F}$ then $\tilde{f}:=\left(f\wedge1\right)_{+}=\max\left\{ f\wedge1,0\right\} \in\mathcal{F}$
and $\mathcal{E}\left(\tilde{f}\right)\leq\mathcal{E}\left(f\right)$. Let $\left\{ P_{t}\right\} _{t\geq 0}$ be the heat semigroup in $L^{2}$
associated to the Dirichlet form $\left(\mathcal{E},\mathcal{F}\right)$
(See Section \ref{SecA:Semigroup}). The \emph{heat kernel} $p_M\left(x,y,t\right)$, $x,y\in M$, $t\ge0$ of the Dirichlet form $\left(\mathcal{E},\mathcal{F}\right)$
is 
 the integral kernel of the heat semigroup $\left\{ P_{t}\right\} _{t\geq0}$. That is for any $t>0$, $p_M(x,y,t)$ is non-negative, measurable in $(x,y)\in M\times M$, and for any $f\in L^{2}\left(M,\mu\right)$ we have 
\begin{equation}\label{eq:heat-kernel-def}
P_{t}f\left(x\right)=\int_{M}f\left(y\right)p_M\left(x,y,t\right)d\mu\left(y\right)
\end{equation}
for $\mu-$almost every $x\in M$.
The heat kernel satisfies properties
such as Markovian, symmetry, Chapman-Kolmogorov and approximation
of identity properties. (See \cite{Grigoryan-Hu-2014} for more details)

}


Throughout this paper, we consider heat kernels $p_{M}(x,y,t)$ that satisfy a sub-Gaussian upper bound $\left(\UEF\right)$ (see Definition \ref{def-UE}) and measures $\mu$ that satisfy the volume doubling property (see Definition \ref{def-VD}). These assumptions are commonly encountered in the existing literature.
Let $\{X_t\}_{t\ge0}$  be a Hunt process generated by the Dirichlet form $\left(\mathcal{E},\mathcal{F}\right)$. For a domain $D\subset M$, we denote by $\tau_D$ the first exit time from $D$ of the process $X_t$. 
{   
The associated Dirichlet form for the killed
process is given by $\left(\mathcal{E},\mathcal{F}\left(D\right)\right)$
where $\mathcal{F}\left(D\right)$ is the $\mathcal{E}_{1}$-closure
of $\mathcal{F}\cap C_{c}\left(D\right)$ in $\mathcal{F}$.
We refer the reader to Section \ref{sec:A} for all relevant definitions, assumptions and the general setting. 
}

The  main result of this paper is a sharp estimate on the survival probability of $X_t$ outside of an invisible set $\mathcal{N}\subset M$ (of measure $0$ and almost surely unreachable by $X_t$, see \eqref{eq-prop-acces}).

\begin{thm*}[Simplified Version of Theorem \ref{thm:Prob-Estimate}]\label{MainTheorem:Simplified}
Suppose the following conditions are satisfied: sub-Gaussian upper bounds $\left(\UEF\right)$  for $(M, d, \mu, \mathcal{E})$  and a volume doubling property $\left(\VD\right)$ for $\mu$.

Then there exists a properly exceptional set $\mathcal{N}\subset M$  and a uniform constant
$C_{M}>0$ such that for any $\epsilon\in(0,1)$ and any domain $D\subset M$ satisfying
$\lambda(D)>0$ we have
\begin{equation*}
\sup_{x\in D\backslash\mathcal{N}}\mathbb{P}_{x}\left(\tau_{D}>t\right)\leq  C_M\left(1+\frac{1}{\epsilon} \right)^{d'}e^{-(1-\epsilon ) \lambda(D)t},
\end{equation*}
for any $t>0$, where $d'=\frac{\beta^{\prime}-1}{\beta}\alpha\vee \frac{\alpha}{\beta}$ and $\alpha,\beta,\beta^\prime$ are constants found in the volume doubling $\left(\VD\right)$, the heat kernel estimate $\left(\UEF\right)$ and assumptions \eqref{eq:F-regularity}.

\end{thm*}

As a direct consequence of Theorem \ref{thm:Prob-Estimate} we prove in Corollary \ref{Cor:polynomial-bound} the following estimate
\begin{equation}\label{intro-Est1}
\sup_{x\in D\setminus \mathcal{N}} \P_x(\tau_D>t)  \le K \left(1+\frac{2\lambda(D)}{d'} \,t\right)^{d'}e^{-\lambda(D)t}.
\end{equation}
This bound is an asymptotically sharp spectral upper bound because it has the correct exponential rate of decay as shown in the asymptotic \eqref{intro:asymp} and lower bound \eqref{Intro:Lower}. Sharp estimates for $\mathbb{P}_{x}\left(\tau_{D}>t\right)$
or equivalently $\left\Vert P_{t}\right\Vert _{\infty\to\infty}$ in large time have been studied in the literature for various operators such as elliptic and Schr\"{o}dinger operators in $\mathbb{R}^d$. For example, the estimate of the form
\begin{equation}\label{intro-Est2}
\mathbb{P}_{x}\left(\tau_{D}>t\right)\leq C_{d}\left(1+t\right)^{d/2}e^{-\lambda\left(D\right)t}
\end{equation}
for Brownian motion in Euclidean space is well-known (see \cite[Theorem 3.1.2]{Sznitman-1998} or \cite{Simon-1981}), and have since been improved for various cases.   Our main Theorem gives us sharp bounds for the survival probability in a general setting. When considering Brownian motion in $\mathbb{R}^d$ we have that $\alpha=d$ and $\beta=\beta^\prime=2$. Note that the exponent of $d^\prime$ in \eqref{intro-Est1} matches the exponent of $d/2$ of the polynomial term in the classical result of \eqref{intro-Est2}. When $\beta=\beta^\prime>2$ this situation covers many examples including graphs and fractals that have been extensively studied in the literature. See the examples in Section \ref{sec:fractals}.

The above theorem leads us to several interesting results which we list below.

\begin{enumerate}
\item \textbf{Spectral Bounds for Mean Exit Times} (Proposition  \ref{Cor:main_product_bound}): For any $p>0,$ there exists a constant $C_{p}\geq\Gamma\left(p+1\right)$
such that for every domain $D\subset M$ satisfying $\lambda\left(D\right)>0$,
we have
\[
\Gamma\left(p+1\right)\leq\lambda\left(D\right)^{p}\cdot\esup_{x\in D}\mathbb{E}_{x}\left[\tau_{D}^{p}\right]\leq C_{p}.
\] Further, this implies that for any domain $D\subset M$,
\[
\lambda\left(D\right)>0\text{ if and only if }\esup_{x\in D}\mathbb{E}_{x}\left[\tau_{D}\right]<\infty.
\]

\item \textbf{Lower Bound for Survival Probability} (Proposition \ref{Spectral_Lower_Bound})
Assume the metric measure space $\left(M,d,\mu\right)$ is endowed
with a regular, (not necessarily local) Dirichlet form $\left(\mathcal{E},\mathcal{F}\right)$. Then 
\begin{equation}\label{Intro:Lower}
e^{-\lambda\left(D\right)t}\leq\esup_{x\in D}\mathbb{P}_x\left(\tau_{D}>t\right).
\end{equation}

\item \textbf{Large Time Asymptotic for Survival Probability} (Proposition \ref{thm:Main_Asymptotic}):  Further assume that the Dirichlet form $\left(\mathcal{E},\mathcal{F}(D)\right) $ is irreducible (see Definition \ref{def:irreducible}), then there exists a properly exceptional set $\mathcal{N}\subset M$ such that for every $D\subset M$,
\begin{equation}\label{intro:asymp}
\lim_{t\to\infty}\frac{\log\mathbb{P}_{x}\left(\tau_{D}>t\right)}{t}=-\lambda\left(D\right)
\end{equation}
for every $x\in D\setminus\mathcal{N}$.

\item  { \textbf{Estimate for the Hot Spots Constant for Riemannian manifolds} (Proposition \ref{thm:Hot-Spots}):  Let $M$ be a Riemannian manifold with $\mbox{Ric}\geq 0$. For domains $D\subset M$ , let $\lambda_{1}\left(D\right),\mu_{2}\left(D\right)$
be the first non-trivial Dirichlet and Neumann eigenvalues. Let $\varphi_{2}$
be a Neumann eigenfunction for $\mu_{2}$. Fix $\hs\in\left (0,1\right)$. Let $\mathcal{D}_\hs$ be a class of bounded smooth domains $D\subset M$ such that 
$
\mu_{2}\left(D\right)\leq \hs\lambda_{1}(D),
$
for all $D\in \mathcal{D}_\hs$.
Then there exists a universal constant $C_{\mathcal{D}_\hs}\geq1$ (called the \textit{Hot Spots Constant} for $\mathcal{D}_\hs$) 
such that 
\[
\sup_{x\in D}\varphi_{2}\left(x\right)\leq C_{\mathcal{D}_\hs}\sup_{x\in\partial D}\varphi_{2}\left(x\right).
\]}

Note that the constant $C_{\mathcal{D}_\hs}$  holds uniformly for the class $\mathcal{D}_\hs$. This result is a maximum principle-type theorem for the first nontrivial Neumann eigenfunctions of the Laplace-Beltrami operator.  

\item \textbf{Characterization of heat kernel upper bounds} (Proposition \ref{MainThm2}): Another application of Theorem \ref{thm:Prob-Estimate} is a new set of conditions that {characterize sub-Gaussian heat kernel upper bounds}  $\left(\UEF\right)$: When $F(r)=r^\beta$, we obtain that the  sub-Gaussian upper estimate $\left(\UE\right)$ (see Definition \ref{def-UE})
\begin{equation*}
p_{M}(x,y,t)\leq\frac{\Cue}{V\left(x,t^{1/\beta}\right)}\exp\left(-\cue\left(\frac{d\left(x,y\right)^{\beta}}{t}\right)^{\frac{1}{\beta-1}}\right),
\end{equation*}
is equivalent to a regular growth rate of $\lambda\left(B(x,r)\right)\asymp r^{-\beta}$ for balls, a Faber-Krahn inequality and a Harnack-type inequality for the mean exit time. 

\end{enumerate}

Our proof depends on a precise analysis of the heat kernel and may be generalized and adapted for heat kernels satisfying non-Gaussian bounds, for instance for non-local operators. 
We also remark that the conditions assumed in Theorem \ref{thm:Prob-Estimate} and Proposition \ref{MainThm2}, Proposition   \ref{Cor:main_product_bound} and \ref{thm:Main_Asymptotic}   are general as we do not assume that the heat kernel is {\it a priori} continuous. Also note that we work with Dirichlet forms $\left(\mathcal{E},\mathcal{F}\left(D\right)\right)$ on domains $D$ for which their generators $\mathcal{L}_D$ does not necessarily have a discrete spectrum. 
Hence our results are applicable to general domains that need not to be bounded, or have any boundary regularity, or even necessarily have finite measure. The only assumption needed on domains  for most of our results is that $\lambda\left(D\right)>0$, or equivalently that $\esup_{x\in D}\mathbb{E}_{x}\left[\tau_{D}\right] <\infty$. Examples of domains where $\lambda(D)>0$ but $\mu(D)=\infty$ include the infinite slab $D=(-1,1)\times \mathbb{R}^{d-1}\subset\R^d$ and many others (see  \cite{Boudabra-Markowsky-2020,Mariano-Panzo-2020,Lifshits-Nazarov-2019} for example).  Moreover, due to our characterization in Proposition \ref{Cor:main_product_bound} and Remark \ref{rem:positive-spectrum-ex}, bounded domains and some domains of finite measure always satisfies $\lambda\left(D\right)>0$.

The paper is organized as follows. In Section \ref{sec:A}, we provide the setting and assumptions used throughout the paper. In Section \ref{sec:3-MainResults} we state the main result on the sharp upper bound for the survival probability. In Section \ref{sec:Mean-Lower} we obtain sharp lower bound for the survival probability as well as mean exit time estimates. Section \ref{sec:smallball} is devoted to exit time asymptotics as $t\to \infty$. In Sections \ref{sec:HotSpots} and \ref{sec:3-MainResults-b} we give applications of the spectral upper bound. This includes  obtaining a bound for the Hot Spots constant on Riemannian manifolds (see Section \ref{sec:HotSpots}), and a new characterization of a heat kernel upper bound (see Section \ref{sec:3-MainResults-b}). Lastly in Section \ref{sec:7-Applications} we give examples where our results apply such as Carnot groups, sub-Riemannian manifolds with transverse symmetries, fractals and fractal-like manifolds.


\section{Setting and Assumptions}\label{sec:A}

In this section we describe the setting on a general metric measure space and the assumptions that shall be used throughout the paper. Most of the exposition, definitions and assumptions follow the works of \cite{Besov2-2020} and \cite{Grigoryan-Hu-2014,Kigami-2004}. 

\subsection{Metric measure space}
Let $\left(M,d,\mu\right)$ be a metric measure space. We assume that
$\left(M,d\right)$ is a locally compact separable connected metric
space and that the measure $\mu$ is a Radon measure with full support
on $M$. As usual, we define $B(x,r)$ a ball of radius $r$ centered
at $x$ and set $V\left(x,r\right):=\mu\left(B\left(x,r\right)\right)$. We assume that $0<V(x,r)<\infty$ for all $x\in M$ and all $r>0$. 

\begin{df}[\textbf{VD}]\label{def-VD} The metric measure space $(M,d,\mu)$ is said to satisfy the  \emph{volume doubling property} if there is a constant $C_{VD}\geq1$
such that
\begin{equation}
V(x,2r)\leq C_{VD}V(x,r)\label{eq:VD}
\end{equation}
for all $x\in M$ and $r>0$. 
\end{df}
It is known that (\textbf{VD}) implies that there exists $\alpha>0$ and $C_\alpha>0$ such
that 
\begin{equation}
\frac{V(x,R)}{V(y,r)}\leq C_{\alpha}\left(\frac{d(x,y)+R}{r}\right)^{\alpha}\label{eq:Vd-alpha}
\end{equation}
for all $x,y\in M$ and for all $0<r\leq R$ (see \cite[Proposition 5.1]{Grigoryan-Hu-2014}). We will mostly use this interpretation of volume doubling in our estimates.  For this reason we define the following equivalent condition:

\begin{df}[$\VD$]\label{def-VD-alpha} The metric measure space $(M,d,\mu)$ is said to satisfy the volume doubling property with parameter $\alpha$ if \eqref{eq:Vd-alpha} holds for all $x,y\in M$ and for all $0<r\leq R$. 
\end{df}

 \subsection{Heat semigroup, heat kernel and the bottom of the spectrum}\label{SecA:Semigroup}
Assume  a local regular Dirichlet form $\left(\mathcal{E},\mathcal{F}\right)$ on $L^{2}\left(M,\mu\right)$. Recall that a Dirichlet form is {local} if $\mathcal{E}\left(f,g\right)=0$ for
any $f,g\in\mathcal{F}$ with disjoint compact supports; And the Dirichlet form is said to be {regular} if $\mathcal{F}\cap C_{0}\left(M\right)$
is dense   {in $\mathcal{F}$ with respect to $\mathcal{E}_{1}$-norm and in $C_{0}\left(M\right)$
with respect to uniform norm.}  Here $C_{0}(M)$
denotes the space of continuous functions with compact support in $M$
 equipped with the sup-norm. 
Let $\left\{ P_{t}\right\} _{t\geq0}$ be the associated strongly continuous semigroup. Let $\mathcal{L}$ be the generator of the heat semigroup. It is  a non-negative definite self-adjoint operator 
on $L^{2}\left(M,\mu\right)$ such that 
\[
\mathcal{E}\left(f,g\right)=\left(\mathcal{L}f,g\right)
\]
{   for $f\in\mathcal{D}\left(\mathcal{L}\right)$, and $g\in\mathcal{F}$.} 
 Let us emphasize that we do not assume {\it a priori} that
$\left\{ P_{t}\right\} _{t\geq0}$ is ultracontractive. Recall that ultracontractivity means the operator can be extended to a bounded operator from $L^{2}(M,\mu)$
to $L^{\infty}(M,\mu)$ hence admits a heat kernel $p_{M}$ for any $t>0$.

In order to introduce the rest of the assumptions that are needed in this paper, from now on we fix a \emph{parameter function} $F:\left(0,\infty\right)\to\left(0,\infty\right)$
that is a continuous increasing bijection in $(0,\infty)$ satisfying
\begin{equation}
C_{F}^{-1}\left(\frac{R}{r}\right)^{\beta}\leq\frac{F\left(R\right)}{F\left(r\right)}\leq C_F\left(\frac{R}{r}\right)^{\beta^{\prime}},\quad 0<r\leq R \label{eq:F-regularity}
\end{equation}
 for some $1\leq\beta\leq\beta^{\prime}$ and $C_F\geq 1$. Moreover, define
\begin{equation}\label{eq:R}
\mathcal{R}(t):=F^{-1}\left(t\right).
\end{equation}
It is easy to see that $\mathcal{R}$ satisfies the following,
\begin{equation}
C_{F}^{-1/\beta^{\prime}}\left(\frac{T}{t}\right)^{1/\beta^{\prime}}\leq\frac{\mathcal{R}\left(T\right)}{\mathcal{R}\left(t\right)}\leq C_{F}^{1/\beta}\left(\frac{T}{t}\right)^{1/\beta},\quad 0<t\leq T.\label{eq:R-regularity}
\end{equation}

The following assumptions will only be  stated in the theorems when they are assumed. 

\begin{df}[$\UEF$]\label{def-UE}
The metric measure space $(M, d,\mu)$ endowed with a Dirichlet form $(\cE, \cF)$ is said to satisfy $\left(\UEF\right)$ if it admits a heat kernel $p_M$ which satisfies a \emph{sub-Gaussian upper estimate} with parameter function $F$. Namely there exists constants $\Cue,\cue>0$ such that  
\begin{equation}
p_{M}\left(x,y,t\right)\leq\frac{C_{\text{UE}}}{V\left(x,\mathcal{R}\left(t\right)\right)}\exp\left(-\frac{1}{2}t\Phi\left(c_{\text{UE}}\frac{d\left(x,y\right)}{t}\right)\right),\label{UE}
\end{equation}
for all $t>0$ and for almost every $x,y\in M$, where 
\[
\Phi\left(s\right)=\sup_{r>0}\left\{ \frac{s}{r}-\frac{1}{F\left(r\right)}\right\}.
\]
We say $\left(\UE\right)$ is satisfied when $F(r)=r^{\beta}$. In this case, the estimate becomes
\begin{equation}
p_{M}(x,y,t)\leq\frac{\Cue}{V\left(x,t^{1/\beta}\right)}\exp\left(-\cue\left(\frac{d\left(x,y\right)^{\beta}}{t}\right)^{\frac{1}{\beta-1}}\right).\label{UEb}
\end{equation}
\end{df}

Note that $\beta$ here is often called the walk dimension in many literature on fractals. 
Moreover, note that we do not require continuity of the heat kernel in our proofs. 

\begin{df}[$\ULE$]\label{def-ULE}
The metric measure space $(M, d,\mu)$ endowed with a Dirichlet form $(\cE, \cF)$ is said to satisfy $\left(\ULE\right)$  if it admits a  heat kernel  $p_{M}\left(x,y,t\right)$ and satisfies a
\emph{sub-Gaussian upper and lower estimates} with parameter $\beta>1$. Namely there exist  $\Cue,\Cle>0$ and $\cle\geq \cue>0$ such that 
\begin{align}
\frac{\Cle}{V\left(x,t^{1/\beta}\right)}\exp\left(-\cle\left(\frac{d\left(x,y\right)^{\beta}}{t}\right)^{\frac{1}{\beta-1}}\right) & \leq p_{M}\left(x,y,t\right)\nonumber \\
 & \leq\frac{\Cue}{V\left(x,t^{1/\beta}\right)}\exp\left(-\cue\left(\frac{d\left(x,y\right)^{\beta}}{t}\right)^{\frac{1}{\beta-1}}\right)\label{ULE}
\end{align}
for all $t>0$ and almost every $x,y\in M$.
\end{df}
Clearly $\left(\ULE\right)$ implies $\left(\UEF\right)$ with $F(r)=r^{\beta}$. Many of our examples in Section \ref{sec:7-Applications} satisfy $\left(\ULE\right)$.


For any non-empty open set $D\subset M$, let $C_{0}(D)$ be the space of
continuous functions with compact support in $D$ and equipped with the sup-norm. 
Let $\mathcal{F}(D)$ be the closure of $\mathcal{F}\cap C_{0}\left(D\right)$
in the norm of $\mathcal{F}$. We denote by $\mathcal{L}_{D}$
the generator of $\left(\mathcal{E},\mathcal{F}(D)\right)$, and by $P_{t}^{D}$ the related heat semigroup. The \textit{bottom of the spectrum} of $\mathcal{L}_D$ is then defined by
\begin{equation}\label{eq-bott-spec}
\lambda(D):=\inf\text{spec}\mathcal{L}_{D}=\inf_{u\in\mathcal{F}(D)\backslash\{0\}}\frac{\mathcal{E}(u)}{\left\Vert u\right\Vert _{2}^{2}}.
\end{equation}
When $\mathcal{L}_D$ has a discrete spectrum,  $\lambda(D)$  is the {first Dirichlet eigenvalue}.
Clearly $\lambda(D)\geq0$ and domain monotonicity holds for $\lambda(D)$.

\subsection{Diffusion and exit times}\label{Sec:Diff}
Given any Dirichlet metric measure space $(M, d, \mu, \mathcal{E})$, it is known that there always exist a {Hunt process} associated with $\left(\mathcal{E},\mathcal{F}\right)$ on $L^{2}\left(M,\mu\right)$ (see \cite{Fukushima-book-ed1-1994} for more details on Hunt processes). Note that the Hunt process is not unique. Therefore we shall fix one process and denote it by $\left(\left\{ X_{t}\right\} _{t>0},\left\{ \mathbb{P}_{x}\right\} _{x\in M}\right)$ for the rest of this paper. Given that  $\left(\mathcal{E},\mathcal{F}\right)$ is local,  the Hunt process  $X_t$, $t>0$ is a {diffusion}, namely  the sample paths $t\mapsto X_t$ are continuous almost surely. Given any bounded Borel function $f$ on $M$,  set 
\[
\mathcal{P}_{t}f\left(x\right):=\mathbb{E}_{x}\left[f\left(X_{t}\right)\right],\,\,\,\,\,t>0, \quad x\in M
\]
Note that $\mathcal{P}_{t}f$ is defined pointwise for all $x\in M$.  The connection between
$P_{t}$ and $\mathcal{P}_{t}$ is that 
\[
    P_{t}f\left(x\right)=\mathcal{P}_{t}f\left(x\right)
\]
holds for $\mu$--almost every $x\in M$. 

We say that  $\left(\mathcal{E},\mathcal{F}\right)$
is {conservative} (or equivalently $X_t$, $t\ge0$ is {stochastically complete}) if and only if 
\[
\mathbb{P}_{x}\left(X_{t}\in M\right)=1
\]
for all $t>0$ and $x\in M$, so that there is no cemetery point. 

We say that a Borel set $\mathcal{N}\subset M$ is \textit{properly
exceptional} if for all $x\in M\backslash\mathcal{N}$,
\begin{equation}\label{eq-prop-acces}
\mu\left(\mathcal{N}\right)=0\quad\text{and}\quad \mathbb{P}_{x}\left(X_{t}\in\mathcal{N}\text{ for some }t\geq0\right)=0.
\end{equation}
Sometimes such sets are called \textit{invisible} sets. (see \cite{Fukushima-book-ed1-1994,Grigoryan-Hu-2014})

For an open set $D\subset M$, define the {first exit time
}from $D$ (or the {lifetime} of $X_{t}$ in $D$) by 
\[
\tau_{D}:=\inf\left\{ t>0:X_{t}\notin D\right\} .
\]
We call $\left(\left\{ X_{t}\right\} _{t\ge0},\left\{ \mathbb{P}_{x}\right\} _{x\in M}\right)$ the Hunt process associated with the Dirichlet form $\left(\mathcal{E},\mathcal{F}\left(D\right)\right)$   {which is regular due to \cite[Theorem 3.3]{Grigoryan-Hu-Lau-2014}} (see also  \cite{Fukushima-book-ed1-1994}),
by imposing killing upon exiting $D$. This process has a transition
function  which is given by 
$
\mathbb{P}_{x}\left(X_{t}\in A,\tau_{D}>t\right)
$
where $x\in M$ and $A\subset M$.  For any bounded Borel function
$f$, we can thus define the semigroup operator
\begin{equation}\label{eq:prob-killed-heat-kernel}
\mathcal{P}_{t}^{D}f\left(x\right):=\mathbb{E}_{x}\left[f\left(X_{t}\right)1_{\left(\tau_{D}>t\right)}\right],\quad x\in M,\ t>0,    
\end{equation}
which can be extended for all bounded or non-negative Borel functions. In fact, for such functions it holds that
\begin{align}\label{eq-ident}
P_{t}^{D}f\left(x\right)=\mathcal{P}_{t}^{D}f\left(x\right)
\end{align}
for $\mu-$almost every $x\in M$. 

{    
\begin{remark}
We point out that the equality \eqref{eq-ident} plays a key role in the proof of our main Theorem \ref{thm:Prob-Estimate}, in which we use an almost sure heat kernel bound $\UEF$, see \eqref{UE}, to give estimates for an everywhere defined quantity $\mathbb{P}_x(\tau_D>t)$.
The gap between these two definitions of the heat semigroups $P_{t}^D$ and $\mathcal{P}_t^D$,  
are addressed in Step 3 of the proof of Theorem \ref{thm:Prob-Estimate}. 
\end{remark}

}

\section{Sharp Upper Bound for the Survival Probability \label{sec:3-MainResults}}


In this section we present the main result that provides a sharp spectral estimate on the survival probability of $X_t, t\ge0$ inside a given domain. 
Results regarding the large time behavior of $\mathbb{P}_{x}\left(\tau_{D}>t\right)$
or $\left\Vert P_{t}\right\Vert _{\infty\to\infty}$ for large $t$
have been deeply studied in the literature. For example, it is well
known that for Brownian motion in $\R^d$, the survival probability can be bounded
by $C\left(1+\lambda (D) t \right)^{d/2}e^{-\lambda\left(D\right)t}$ (see \cite[Theorem 3.1.2]{Sznitman-1998} or \cite{Simon-1981}). Sharper results for large time $t>0$ have since been obtained for operators in the Euclidean setting such as in \cite{Ouhabaz-2006,Vogt-2019a}. For example, the bound of $C\left(1+\lambda (D) t\right)^{d/4}e^{-\lambda\left(D\right)t}$ was  obtained in \cite{Vogt-2019a} for operators in $\mathbb{R}^d$ satisfying Gaussian bounds. As mentioned in the introduction, our result will match  at least the classical bounds of $C\left(1+t\right)^{d/2}e^{-\lambda\left(D\right)t}$ for elliptic operators, which were not previously known in this generality.


\begin{theorem}
\label{thm:Prob-Estimate}
{Let $\alpha, \alpha'>0$ and $F$ be a parameter function satisfying \eqref{eq:F-regularity} with $\beta,\beta^{\prime}>0$}. Assume the metric measure space $\left(M,d,\mu\right)$ is endowed with a regular, local Dirichlet form $\left(\mathcal{E},\mathcal{F}\right)$, and satisfies  $\left(\VD\right)$ and $\left(\UEF\right)$.

Then there exists a properly exceptional set $\mathcal{N}\subset M$ and   a constant $C_M>0$ such that for every $\epsilon\in\left(0,1\right)$ 
and for all open sets $D\subset M$ satisfying $\lambda\left(D\right)>0$ we have
\begin{equation}\label{eq:Probet-N-set1}
\sup_{x\in D\backslash\mathcal{N}}\mathbb{P}_{x}\left(\tau_{D}>t\right)\leq  C_M\left(1+\frac{1}{\epsilon} \right)^{d'}e^{-(1-\epsilon ) \lambda(D)t}
\end{equation}
for any $t>0$, where $d'=\frac{\beta^{\prime}-1}{\beta}\alpha\vee \frac{\alpha}{\beta}$. The constant $C_M$ depends only on $ \alpha,\beta,\beta^{\prime}$ and the constants that appear in the assumptions.
\end{theorem}

\begin{corollary}\label{Cor:polynomial-bound}
Assume the metric measure space $\left(M,d,\mu\right)$ satisfies the assumptions as in Theorem \ref{thm:Prob-Estimate}. 
Then for all $t>0$, and all open sets $D\subset M$  satisfying $\lambda\left(D\right)>0$, 
\[
\sup_{x\in D\setminus \mathcal{N}} \P_x(\tau_D>t)  \le K \left(1+\frac{2\lambda(D)}{d'} \,t\right)^{d'}e^{-\lambda(D)t}.
\]
where $d'=\frac{\beta^\prime-1}{\beta}\alpha$ and $K>0$ is a constant that depends only on $ \alpha,\beta,\beta^{\prime}$ and the constants appear in the assumptions. 

\end{corollary}

We devote the rest of this section into the proofs of Theorem \ref{thm:Prob-Estimate} and Corollary \ref{Cor:polynomial-bound}. The proof requires a novel robust technique that accounts for the general setting in the present paper. In particular it covers general (possibly unbounded) domains that do not necessarily have a discrete spectrum. In Section \ref{sec-prelim} we show some lemmas that serve as preliminaries. In Section \ref{sec-proof-main} we present the proof of Theorem \ref{thm:Prob-Estimate}, which is split into four steps. Lastly we present the proof of Corollary \ref{Cor:polynomial-bound} in Section \ref{sec-proof-cor}.



\subsection{Some preliminaries}\label{sec-prelim}

  {
If it exists, the \emph{Dirichlet heat kernel} $p_D\left(x,y,t\right)$, $x,y\in D$, $t\ge0$ of the Dirichlet form $\left(\mathcal{E},\mathcal{F}(D)\right)$
is  the integral kernel of the heat semigroup $\left\{ P_{t}^D\right\} _{t\geq0}$ in the sense of \eqref{eq:heat-kernel-def}.
}

We will need the following preliminary results. 

\begin{lemma}\label{lem:OnDiag-Eige}
Let $D\subset M$ be a domain and suppose its Dirichlet heat kernel $p_{D}$ exists. Assume that 
$p_{D}\left(\cdot,x,t\right)\in L^{2}\left(D\right)$
for a.e. $x\in D$ and $t>0$. Given that $\lambda(D)>0$ we have for any $\epsilon\in(0,1)$ that
\[
\int_{D}p_{D}\left(x,y,t\right)^{2}d\mu\left(y\right)\leq e^{-2\left(1-\epsilon\right)t\lambda(D)}\int_{D}p_{D}\left(x,y,\epsilon t\right)^{2}d\mu\left(y\right),
\]
for almost every $x\in D$. Moreover, if $p_{D}$ is known to be continuous
then for all $x\in D$,
\begin{align}\label{eq-on-diag-pt}
 p_{D}\left(x,x,t\right)\leq e^{-\left(1-\epsilon\right)\lambda(D)t}p_{D}\left(x,x,\epsilon t\right).   
\end{align}

%
\end{lemma}
\begin{proof}
Let $P_t^D$ be the associated heat semigroup and $\left\{ E_{\lambda}^{D}\right\} _{\lambda>0}$ the spectral
resolution  in $L^{2}(D)$. 
Then for any $f\in L^{2}\left(D\right)$ it holds that 
\begin{align*}
P_{t}^{D}f & =\int_{0}^{\infty}e^{-t\lambda}dE_{\lambda}^{D}f
\end{align*}
and
\begin{align}\label{eq-spec-meas}
\left\Vert P_{t}^{D}f\right\Vert _{2}^{2} & =\int_{0}^{\infty}e^{-2t\lambda}d\left\Vert E_{\lambda}^{D}f\right\Vert ^{2}
\end{align}
for all $t>0$. For $\mu-$ almost every $x\in D,s>0$ define the function $f_{s,x}(\cdot):=p_{D}\left(\cdot,x,s\right)$. Then by
Chapman-Kolmogorov we have
\[
P_{t}^{D}f_{s,x}=p_{D}\left(\cdot,x,t+s\right),
\]
for $\mu-$ almost every $x\in D$. Hence by \eqref{eq-spec-meas} we obtain that
\begin{align} \label{eq-spect-meas-a}
\left\Vert p_{D}\left(\cdot,x,t\right)\right\Vert _{2}^{2}=\left\Vert P^D_{ct}f_{\left(1-c\right)t,x}\right\Vert _{2}^{2}
=\int_{0}^{\infty}e^{-2\left(ct\right)\lambda}d\left\Vert E_{\lambda}^{D}f_{\left(1-c\right)t,x}\right\Vert ^{2}
 \end{align}
for any $t>0$ and $c\in (0,1)$.  Since $\lambda(D)>0$,  we have for any $\delta\in\left(0,1\right)$ that
\begin{align} \label{eq-spect-meas-b}
\int_{0}^{\infty}e^{-2\left(ct\right)\lambda}d\left\Vert E_{\lambda}^{D}f_{\left(1-c\right)t,x}\right\Vert ^{2}
&\le e^{-2\left(ct\right)\left(1-\delta\right)\lambda(D)}\int_{\lambda(D)}^{\infty}e^{-2\left(c\delta t\right)\lambda}d\left\Vert E_{\lambda}^Df_{\left(1-c\right)t,x}\right\Vert ^{2} \notag\\
&=e^{-2c\left(1-\delta\right)t\lambda(D)}\left\Vert P^D_{c\delta t}f_{\left(1-c\right)t,x}\right\Vert _{2}^{2} \nonumber\\
& =e^{-2c\left(1-\delta\right)t\lambda(D)}\left\Vert p_{D}\left(\cdot,x,c\delta t+\left(1-c\right)t\right)\right\Vert _{2}^{2}
\end{align}
Let $1-\epsilon=c\left(1-\delta\right)\in(0,1)$ and combine \eqref{eq-spect-meas-a} and \eqref{eq-spect-meas-b} we obtain 
\begin{align*}
\left\Vert p_{D}\left(\cdot,x,t\right)\right\Vert _{2} 
 \le e^{-2\left(1-\epsilon\right)\lambda(D)t}\left\Vert p_{D}\left(\cdot,x,\epsilon t\right)\right\Vert _{2}
\end{align*}
as needed. {     Additionally, if $p_D$ is continuous, then the on-diagonal heat kernel $p_D(t,x,x)$ is well-defined for all $x\in D$. By the symmetry of the heat kernel and using Chapman-Kolmogorov we then obtain \eqref{eq-on-diag-pt}.}
\end{proof}

Next we recall a general version of an integration by parts formula in a metric measure space setting and sketch its proof.

\begin{lemma}
\label{prop:Change-of-Variables} Let $\left(M,\mu,d\right)$ be a
metric measure space. Suppose $\varphi:[0,\infty)\to[0,\infty)$ is a
decreasing function that satisfies $\varphi(\infty)=0$ and $\varphi\in C^{1}\left([0,\infty)\right)$,
then 
\[
\int_{B(x,R)^{c}}\varphi\left(d\left(x,y\right)\right)d\mu\left(y\right)=-\varphi(R)\mu\left(B\left(x,R\right)\right)-\int_{R}^{\infty}\mu\left(B\left(x,r\right)\right)\varphi^{\prime}(r)dr.
\]
\end{lemma}

\begin{proof}
By Fubini's theorem, we have 
\begin{align*}
\int_{B(x,R)^{c}}\varphi\left(d\left(x,y\right)\right)d\mu\left(y\right) & =\int_{B(x,R)^{c}}\int_{0}^{\varphi\left(d\left(x,y\right)\right)}dtd\mu\left(y\right)\\
 & =\int_{0}^{\varphi\left(R\right)}\mu\left(y\in B(x,R)^{c}:\varphi\left(d\left(x,y\right)\right)>t\right)dt\\
 & =\int_{0}^{\varphi\left(R\right)}\mu\left(y:\varphi\left(d\left(x,y\right)\right)>t\right)dt-\int_{0}^{\varphi\left(R\right)}\mu\left(y\in B(x,R):\varphi\left(d\left(x,y\right)\right)>t\right)dt.
\end{align*}
Since 
\[
\int_{0}^{\varphi\left(R\right)}\mu\left(y\in B(x,R):\varphi\left(d\left(x,y\right)\right)>t\right)dt=\int_{0}^{\varphi\left(R\right)}\mu\left(y\in B(x,R)\right)dt=\varphi(R)\mu\left(B\left(x,R\right)\right),
\]
we have
\begin{align*}
\int_{B(x,R)^{c}}\varphi\left(d\left(x,y\right)\right)d\mu\left(y\right) & =\int_{0}^{\varphi\left(R\right)}\mu\left(y:\varphi\left(d\left(x,y\right)\right)>t\right)dt-\varphi(R)\mu\left(B\left(x,R\right)\right)\\
 & =\int_{0}^{\varphi\left(R\right)}\mu\left(B\left(x,\phi\left(t\right)\right)\right)dt-\varphi(R)\mu\left(B\left(x,R\right)\right)
\end{align*}
where $\phi(t)=\varphi^{-1}(t)$ and $0<t<\varphi\left(R\right)$. Note
that that $\phi^{\prime}(t)=\frac{1}{\varphi^{\prime}\left(\phi(t)\right)}$
so that by change of variables with $r=\phi(t)$ we have $dr=\frac{1}{\varphi^{\prime}\left(\phi(t)\right)}dt$
hence $dt=\varphi^{\prime}\left(\phi(t)\right)dr=\varphi^{\prime}\left(r\right)dr$.
Thus 
\begin{align*}
\int_{0}^{\varphi\left(R\right)}\mu\left(B\left(x,\phi\left(t\right)\right)\right)dt & =\int_{\infty}^{R}\mu\left(B\left(x,r\right)\right)\varphi^{\prime}\left(r\right)dr\\
 & =-\int_{R}^{\infty}\mu\left(B\left(x,r\right)\right)\varphi^{\prime}\left(r\right)dr.
\end{align*}
This completes the proof. 
\end{proof}

Recall that a function $u:M\to\mathbb{R}$ is called \textit{quasi-continuous}
if it is continuous in $M\backslash E$ for some set $E$ of capacity
$0$. We recall the following proposition from \cite{Grigoryan-Hu-2014} that will be useful later.
\begin{prop}[Proposition 6.1, \cite{Grigoryan-Hu-2014}]\label{GH-Prop6.1}
Let $u$ be a quasi-continuous function on $M$
such that $u\geq0$ $\mu-$ a.e. then there is a properly exceptional
set $\mathcal{N}$ such that $u\left(x\right)\geq0$ for all $x\in M\backslash\mathcal{N}$. 
\end{prop}

 \subsection{Proof of Theorem \ref{thm:Prob-Estimate}} \label{sec-proof-main}
We complete the proof by 4 steps.

%


\textbf{\underline{Step 1:}} 

Let us start with the existence of the heat kernel $p_D$ of the analytic heat semigroup $(P_t^D)_{t\ge0}$  for any domain $D\subset M$ such that $\mu \left(D\right)<\infty$. To see this, we first need the ultracontractivity of the heat semigroup $P_t^B$ for any ball $B\subset M$. Precisely, we want to show that there
is a positive left-continuous function $M\left(t\right)$ such
that for any $f\in L^{1}\cap L^{2}\left(B\right)$ and $t>0$ one
has 
\begin{align}\label{eq-ultra}
    \left\Vert P_{t}^{B}f\right\Vert _{\infty}\leq M\left(t\right)\left\Vert f\right\Vert _{1}.
\end{align}


Using Markov property of $P_t^B$ and 
$\left(\UEF\right)$ we see that  
\begin{align*}
\left\Vert P_{t}^{B}f\right\Vert _{\infty} & =\left\Vert P_{t}^{B}\left(f1_{B}\right)\right\Vert _{\infty}\leq\left\Vert P_{t}^{B}\left(\left|f\right|1_{B}\right)\right\Vert _{\infty}\leq\left\Vert P_{t}\left(\left|f\right|1_{B}\right)\right\Vert _{\infty}\\
 & =\esup_{x\in B}\int_{B}\left|f\left(y\right)\right|p_M\left(x,y,t\right)dy\leq\esup_{x\in B}\frac{C_{\text{UE}}}{V\left(x,\mathcal{R}\left(t\right)\right)}\left\Vert f\right\Vert _{1}.
\end{align*}
 By ($\VD$) and \eqref{eq:R-regularity} we have that there exists a constant $C>0$
such that for any $t>0$, 
\begin{align*}
\frac{1}{V\left(x,\mathcal{R}\left(t\right)\right)} & \leq\frac{1}{V\left(x,\mathcal{R}\left(1\right)\right)}\left(C_{\alpha}\left(\frac{\mathcal{R}\left(1\right)}{\mathcal{R}\left(t\right)}\right)^{\alpha}1_{\left(t<1\right)}+1\right)\leq\frac{1}{V\left(x,\mathcal{R}\left(1\right)\right)}\left(\frac{C}{t^{\alpha/\beta}}+1\right).
\end{align*}
Fix any $x_{0}\in B$. By \eqref{eq:Vd-alpha} there exists a constant
$C_{\text{diam}\left(B\right)}$ depending on the $\text{\text{diam}\ensuremath{\left(B\right)}}$
such that 
\[
\frac{V\left(x_{0},\mathcal{R}\left(1\right)\right)}{V\left(x,\mathcal{R}\left(1\right)\right)}\leq C\left(\frac{d\left(x,x_{0}\right)}{\mathcal{R}\left(1\right)}+1\right)^{\alpha}\leq C_{\text{diam}\left(B\right)}.
\]
Putting these estimates together shows that \eqref{eq-ultra}
holds with $M\left(t\right)=\frac{C_{\text{diam}\left(B\right)}}{V\left(x_{0},\mathcal{R}\left(1\right)\right)}\left(\frac{C}{t^{\alpha/\beta}}+1\right)$, which is positive and continuous. 
Therefore $P_{t}^{B}$ is ultracontractive for any ball $B\subset M$, which by definition means $P_t$ is locally ultracontractive.
Due to \cite[Theorem 2.12]{Grigoryan-Telcs-2012} we know that there exists a properly exceptional set $\mathcal{N}\subset M$
such that for every $D\subset M$, the semigroup $\mathcal{P}_{t}^{D}$
defined in \eqref{eq:prob-killed-heat-kernel} has an integral kernel $p_{D}\left(\cdot,\cdot,t\right)$
in the strong sense as \cite[Definition 2.3.]{Grigoryan-Telcs-2012} with domain $D\backslash \mathcal{N}$. By \cite[Equation (2.7)]{Grigoryan-Telcs-2012} one has that $p_D$ is a heat kernel for $P_{t}^{D}$.

Next we claim that under the assumptions in Theorem \ref{thm:Prob-Estimate},
\begin{align}\label{eq-claim}
\esup_{x\in D} P_t 1_{D}(x)=\esup_{x\in D}\int_{D}p_{D}\left(x,y,t\right)d\mu(y)\le
 C_{\pa,\delta}\, e^{-\pa(1-\delta ) t\lambda(D)}
\end{align}
for any $\pa, \delta\in (0,1)$ with $C_{\pa,\delta}>0$ that depends on $\pa,\delta,\alpha, \beta$ and constants appear in the assumptions.

For any $\pa\in (0,1)$, by H\"older's inequality with $p=\frac2\pa$ and $q=\frac{2}{2-\pa}$ we have for a.e. $x\in D$
\begin{align}\label{eq-int-pd}
\int_{D}p_{D}\left(x,y,t\right)d\mu(y)&=\int_{D}p_{D}\left(x,y,t\right)^{\pa}p_{D}\left(x,y,t\right)^{1-\pa}d\mu(y)\notag\\
&\le 
\left(\underbrace{\int_{D}p_{D}\left(x,y,t\right)^{2}d\mu\left(y\right)}_{I}\right)^{\pa/2}  \left(\underbrace{\int_{D}p_{D}\left(x,y,t\right)^{\frac{2(1-\pa)}{2-\pa}}d\mu\left(y\right)}_{II}\right)^{\frac{2-\pa}{2}}
\end{align}
\textbf{\underline{Estimate for $I$:}}
From $\left(\UEF\right)$ we know that for a.e. $x\in D$ and $t>0$, 
\begin{equation}\label{eq:usespectrallem}
\left\Vert p_{D}\left(\cdot,x,t\right)\right\Vert _{2}^{2}
\leq \frac{C_{\text{UE}}^{2}\mu\left(D\right)}{V\left(x,\mathcal{R}\left(t\right)\right)^{2}}<\infty,
\end{equation}
which shows that $p_{D}\left(\cdot,x,t\right)\in L^{2}\left(D\right)$. Then by  Lemma \ref{lem:OnDiag-Eige} we have that 
\begin{align*}
I&\le  e^{-2\left(1-\delta\right)t\lambda(D)}\int_{D}p_{D}\left(x,y,\delta t\right)^{2}d\mu\left(y\right) \\
&\le  e^{-2\left(1-\delta\right)t\lambda(D)}  \esup_{y\in D}p_{D}\left(x,y,\delta t\right) \int_{D}p_{D}\left(x,y,\delta t\right)d\mu\left(y\right) \\
&\le  e^{-2\left(1-\delta\right)t\lambda(D)}  \esup_{y\in D}p_{D}\left(x,y,\delta t\right).
\end{align*}
Moreover since $\left(\UEF\right)$ implies that 
\begin{equation}\label{eq-pD-essup}
\esup_{y\in D}p_{D}(x,y,\delta t)\leq\frac{\Cue}{V\left(x,\mathcal{R}\left(\delta t\right)\right)},
\end{equation}
we then obtain that
\begin{equation}\label{Estimate_on_I}
I\le e^{-2\left(1-\delta\right)t\lambda(D)} \frac{\Cue}{V\left(x,\mathcal{R}\left(\delta t\right)\right)}.
\end{equation}

\textbf{\underline{Estimate for $II$:}}
 Next we estimate $II$  under the assumptions of $\left(\VD\right)$ and $\left(\UEF\right)$ from Theorem  \ref{thm:Prob-Estimate}. \\

 Recall the notation $q=\frac{2}{2-\pa}$. {     Since the power of $p_D(x,y,t)$ in $II$ satisfies $0<\frac{2(1-\gamma)}{2-\gamma}<1$.} Applying $\left(\UEF\right)$ we obtain that
 \begin{equation}\label{eq-II-a}
II\leq\int_{D}\frac{C_{\text{UE}}^{\left(1-\pa\right)q}}{V\left(x,\mathcal{R}\left(t\right)\right)^{\left(1-\pa\right)q}}\exp\left(-\frac{\left(1-\pa\right)q}{2}t\Phi\left(c_{\text{UE}}\frac{d\left(x,y\right)}{t}\right)\right)d\mu\left(y\right).
 \end{equation}
By  \cite[Lemma 3.19]{Grigoryan-Telcs-2012}, there exists a constant $c_{1}>0$ such
that for all $R,t>0$, \begin{equation}
t\Phi\left(\frac{R}{t}\right)\geq c_{1}\min\left\{ \left(\frac{F\left(R\right)}{t}\right)^{1/\left(\beta^{\prime}-1\right)},\left(\frac{F\left(R\right)}{t}\right)^{1/\left(\beta-1\right)}\right\}.\label{eq:Phi-lower-bound}
\end{equation}
Plug the above estimates  into \eqref{eq-II-a} we
obtain 
 \begin{equation}
II\leq\int_{M}H\left(d\left(x,y\right)\right)d\mu\left(y\right),\label{eq:I-F-estimate-1}
\end{equation}
where 
\begin{align*}
H\left(r\right) & :=\frac{C_{\text{UE}}^{\left(1-\pa\right)q}}{V\left(x,\mathcal{R}\left(t\right)\right)^{\left(1-\pa\right)q}}\exp\left(-c_{\text{UE}}c_{1}\frac{\left(1-\pa\right)q}{2}\min\left\{ \left(\frac{F\left(r\right)}{t/c_{\text{UE}}}\right)^{\frac{1}{\beta^{\prime}-1}},\left(\frac{F\left(r\right)}{t/c_{\text{UE}}}\right)^{\frac{1}{\beta-1}}\right\} \right).
\end{align*}
Let 
\begin{align}
A_{\beta^{\prime}} & :=\left\{ y\in M\mid\left(\frac{F\left(d\left(x,y\right)\right)}{t/c_{\text{UE}}}\right)^{\frac{1}{\beta^{\prime}-1}}\leq\left(\frac{F\left(d\left(x,y\right)\right)}{t/c_{\text{UE}}}\right)^{\frac{1}{\beta-1}}\right\} , \quad A_{\beta}  :
=A_{\beta^{\prime}}^{c},\nonumber 
\end{align}
and 
\begin{align}
B_{1} & :=\left\{ y\in M\mid d(x,y)\leq\mathcal{R}\left(t\right)\right\} , \quad B_{2}:=B_{1}^{c},\nonumber 
\end{align}
We then have
\begin{align}\label{eq-II-a-2}
  II & \le \int_{A_{\beta^{\prime}}\cap B_1 }H\left(d\left(x,y\right)\right)d\mu\left(y\right)+\int_{A_{\beta^{\prime}}\cap B_2}H\left(d\left(x,y\right)\right)d\mu\left(y\right)\nonumber \\
 & +\int_{A_{\beta}\cap B_1}H\left(d\left(x,y\right)\right)d\mu\left(y\right)+\int_{A_{\beta}\cap B_2}H\left(d\left(x,y\right)\right)d\mu\left(y\right)\nonumber \\
&=:  II_{\beta^{\prime},1}+II_{\beta^{\prime},2}
+  II_{\beta,1}+II_{\beta,2}.
\end{align}

Using the definitions of $A_{\beta^{\prime}}$ and $A_{\beta}$ we clearly have that for $i=1,2$,
\begin{equation}
II_{\beta^{\prime},i}\leq\int_{A_{\beta^{\prime}}\cap B_{i}} \frac{C_{\text{UE}}^{\left(1-\pa\right)q}}{V\left(x,\mathcal{R}\left(t\right)\right)^{\left(1-\pa\right)q}}\exp\left(-c_{\text{UE}}c_{1}\frac{\left(1-\pa\right)q}{2}\left(\frac{F\left(d(x,y)\right)}{t/c_{\text{UE}}}\right)^{\frac{1}{\beta'-1}}\right)d\mu\left(y\right),\label{eq:II-beta-prime-i}
\end{equation}
and
\begin{equation}
II_{\beta,i}\leq\int_{A_{\beta}\cap B_{i}}\frac{C_{\text{UE}}^{\left(1-\pa\right)q}}{V\left(x,\mathcal{R}\left(t\right)\right)^{\left(1-\pa\right)q}}\exp\left(-c_{\text{UE}}c_{1}\frac{\left(1-\pa\right)q}{2}\left(\frac{F\left(d(x,y)\right)}{t/c_{\text{UE}}}\right)^{\frac{1}{\beta-1}}\right)d\mu\left(y\right).\label{eq:II-beta-i}
\end{equation}
We now estimate the $II_{\beta',i}$ and $II_{\beta,i}$ for $i=1,2$ term by term.

Since $d(x,y)\leq\mathcal{R}\left(t\right)$ on $B_{1}$ and 
 by $(\ref{eq:F-regularity})$ we have
\[
C_{F}^{-1}\left(\frac{\mathcal{R}\left(t\right)}{d\left(x,y\right)}\right)^{\beta}F\left(d\left(x,y\right)\right)\leq F\left(\mathcal{R}\left(t\right)\right)\leq F\left(d\left(x,y\right)\right)C_{F}\left(\frac{\mathcal{R}\left(t\right)}{d\left(x,y\right)}\right)^{\beta^{\prime}}.
\]
Using this in (\ref{eq:II-beta-prime-i}) we have that
\begin{align}
 & II_{\beta^{\prime},1} 
\leq  \frac{C_{\text{UE}}^{\left(1-\pa\right)q}}{V\left(x,\mathcal{R}\left(t\right)\right)^{\left(1-\pa\right)q}}\int_{M}\exp\left(-b(t)d\left(x,y\right)^{\frac{\beta^{\prime}}{\beta^{\prime}-1}}\right)d\mu\left(y\right)\label{eq:Case1a-first}
\end{align}
where $b(t)=\frac{1}{2}c_{\text{UE}}c_{1}\left(1-\pa\right)q\left(\frac{C_{F}^{-1}c_{\text{UE}}F\left(\mathcal{R}\left(t\right)\right)}{\mathcal{R}\left(t\right)^{\beta^{\prime}}t}\right)^{\frac{1}{\beta^{\prime}-1}}$. Using Lemma \ref{prop:Change-of-Variables} with $\varphi\left(r\right)=\exp\left(-b(t)r^{\frac{\beta^{\prime}}{\beta^{\prime}-1}}\right)$ and $R=0$
we have,
\begin{align}
&\frac1{V\left(x,\mathcal{R}\left(t\right)\right)}\int_{M}\exp\left(-b(t)d\left(x,y\right)^{\frac{\beta^{\prime}}{\beta^{\prime}-1}}   \right)d\mu\left(y\right)=-\frac1{V\left(x,\mathcal{R}\left(t\right)\right)}\int_{0}^{\infty}\mu\left(B\left(x,r\right)\right)\varphi^{\prime}(r)dr\notag\\
&\quad\quad
=b(t)\frac{\beta^\prime}{\beta^\prime-1}\int_{0}^{\infty}\frac{V\left(x,r\right)}{V\left(x,\mathcal{R}\left(t\right)\right)}\exp\left(-b(t)r^{\frac{\beta^{\prime}}{\beta^{\prime}-1}}\right)r^{\frac{1}{\beta^{\prime}-1}}dr. \label{eq-mid-1}
\end{align}

Split the above integral into two parts $\int_0^{\mathcal{R}\left(t\right)}$ and $\int_{\mathcal{R}\left(t\right)}^\infty$. Since $V\left(x,r\right)\leq V\left(x,\mathcal{R}\left(t\right)\right)$
for $r\leq\mathcal{R}\left(t\right)$ then  
\begin{align}\label{eq-int-est<R}
\int_{0}^{\mathcal{R}\left(t\right)}\frac{V\left(x,r\right)}{V\left(x,\mathcal{R}\left(t\right)\right)}\exp\left(-b(t)r^{\frac{\beta^{\prime}}{\beta^{\prime}-1}}\right)r^{\frac{1}{\beta^{\prime}-1}}dr & \leq \int_{0}^{\mathcal{R}\left(t\right)}\exp\left(-b(t)r^{\frac{\beta^{\prime}}{\beta^{\prime}-1}}\right)r^{\frac{1}{\beta^{\prime}-1}}dr\\
 & \leq\frac{\beta^{\prime}-1}{\beta^{\prime}}\frac{1}{b(t)}
\end{align}
By assumption $\left(\VD\right)$ we know that there exist a constant $C_{\alpha}$ such that
\begin{align}\label{eq-int-est>R}
\int_{\mathcal{R}\left(t\right)}^{\infty}\frac{V\left(x,r\right)}{V\left(x,\mathcal{R}\left(t\right)\right)}\exp\left(-b(t)r^{\frac{\beta^{\prime}}{\beta^{\prime}-1}}\right)r^{\frac{1}{\beta^{\prime}-1}}dr & \leq C_{\alpha}\int_{\mathcal{R}\left(t\right)}^{\infty}\frac{r^{\alpha}}{\mathcal{R}\left(t\right)^{\alpha}}\exp\left(-b(t)r^{\frac{\beta^{\prime}}{\beta^{\prime}-1}}\right)r^{\frac{1}{\beta^{\prime}-1}}dr\notag\\
 & \leq\frac{C_{\alpha}}{\mathcal{R}\left(t\right)^{\alpha}}\frac{\beta^{\prime}-1}{\beta^{\prime}}\frac{\Gamma\left(1+\frac{\beta^{\prime}-1}{\beta^{\prime}}\alpha\right)}{b(t)^{\alpha\frac{\beta^{\prime}-1}{\beta^{\prime}}+1}}.
\end{align}
Moreover, since $\mathcal{R}\left(t\right)=F^{-1}\left(t\right)$, we have that
\begin{align*}
 & \frac{b(t)^{-\alpha\frac{\beta^{\prime}-1}{\beta^{\prime}}}}{\mathcal{R}\left(t\right)^{\alpha}}
=  \left(\frac{1}{2}c_{\text{UE}}c_{1}\right)^{-\alpha\frac{\beta^{\prime}-1}{\beta^{\prime}}}\left(C_{F}^{-1}c_{\text{UE}}\right)^{-\alpha/\beta^{\prime}}\left(\left(1-\pa\right)q\right)^{-\alpha\frac{\beta^{\prime}-1}{\beta^{\prime}}}.
\end{align*}
Plugging the above  equalities and \eqref{eq-int-est<R}, \eqref{eq-int-est>R} into \eqref{eq-mid-1} we have
\begin{align}
\frac{1}{V\left(x,\mathcal{R}\left(t\right)\right)}\int_{M}\exp\left(-b(t)d\left(x,y\right)^{\frac{\beta^{\prime}}{\beta^{\prime}-1}}\right)d\mu\left(y\right) & \leq 1+ C_M\left(\frac{2-2\pa}{2-\pa}\right)^{-\alpha\frac{\beta^{\prime}-1}{\beta^{\prime}}} \nonumber \\
 & \leq 2C_M\left(\frac{2-2\pa}{2-\pa}\right)^{-\alpha\frac{\beta^{\prime}-1}{\beta^{\prime}}}.\label{eq-mid-3}
\end{align}
From this point we use $C_M$ to denote constants that depend on $\alpha,\alpha^\prime,\beta,\beta^\prime,F$ and other constants in the assumptions. The constant $C_M$ can change from line to line.  Plugging (\ref{eq-mid-3}) into (\ref{eq:Case1a-first}) we obtain
the bound of 
\begin{equation}
II_{\beta^{\prime},1}\leq\frac{C_{\text{UE}}^{\left(1-\pa\right)q}}{V\left(x,\mathcal{R}\left(t\right)\right)^{\left(1-\pa\right)q-1}}C_M\left(\frac{2-2\pa}{2-\pa}\right)^{-\alpha\frac{\beta^{\prime}-1}{\beta^{\prime}}}.\label{eq:Case1a-final}
\end{equation}
Next we estimate $II_{\beta^{\prime},2}$. Since $d(x,y)\geq\mathcal{R}\left(t\right)$ on $B_{2}$, by $(\ref{eq:F-regularity})$ we have
\[
C_{F}^{-1}\left(\frac{d\left(x,y\right)}{\mathcal{R}\left(t\right)}\right)^{\beta}F\left(R\left(t\right)\right)\leq F\left(d\left(x,y\right)\right)\leq F\left(\mathcal{R}\left(t\right)\right)C_{F}\left(\frac{d\left(x,y\right)}{\mathcal{R}\left(t\right)}\right)^{\beta^{\prime}}.
\]
Plugging it into \eqref{eq:II-beta-prime-i} we obtain
\begin{align*}
 & II_{\beta^{\prime},2}
\leq  \frac{C_{\text{UE}}^{\left(1-\pa\right)q}}{V\left(x,\mathcal{R}\left(t\right)\right)^{\left(1-\pa\right)q}}\int_{M}\exp\left(-b(t)d\left(x,y\right)^{\frac{\beta}{\beta^{\prime}-1}}\right)d\mu\left(y\right)
\end{align*}
where $b(t)=\frac{1}{2}c_{\text{UE}}c_{1}\left(1-\pa\right)q\left(\frac{C_{F}^{-1}c_{\text{UE}}}{\mathcal{R}\left(t\right)^{\beta}}\right)^{\frac{1}{\beta^{\prime}-1}}$.
Following similar estimate as before we have
\begin{equation}
II_{\beta^{\prime},2}\leq\frac{C_{\text{UE}}^{\left(1-\pa\right)q}}{V\left(x,\mathcal{R}\left(t\right)\right)^{\left(1-\pa\right)q-1}}C_M\left(\frac{2-2\pa}{2-\pa}\right)^{-\alpha\frac{\beta^{\prime}-1}{\beta}}.\label{eq:Case1b-final}
\end{equation}
Similarly we can also obtain that
\begin{equation}\label{eq:Case1c-final}
II_{\beta,1}\leq\frac{C_{\text{UE}}^{\left(1-\pa\right)q}}{V\left(x,\mathcal{R}\left(t\right)\right)^{\left(1-\pa\right)q-1}}C_M\left(\frac{2-2\pa}{2-\pa}\right)^{-\alpha\frac{\beta-1}{\beta^{\prime}}}\end{equation}
and
\begin{equation}\label{eq:Case1d-final}
II_{\beta,2}\leq\frac{C_{\text{UE}}^{\left(1-\pa\right)q}}{V\left(x,\mathcal{R}\left(t\right)\right)^{\left(1-\pa\right)q-1}}C_M \left(\frac{2-2\pa}{2-\pa}\right)^{-\alpha\frac{\beta-1}{\beta}}.
\end{equation}

Next we plug \eqref{eq:Case1a-final}, \eqref{eq:Case1b-final}, \eqref{eq:Case1c-final}  and \eqref{eq:Case1d-final} into \eqref{eq-II-a-2}. Taking into account that $\beta\leq\beta^{\prime}$ and
$\frac{2-2\pa}{2-\pa}<1$ for $\pa\in\left(0,1\right)$ we have 
\begin{equation}\label{Estimate_on_II}
II\leq\frac{C_{\text{UE}}^{\left(1-\pa\right)q}}{V\left(x,\mathcal{R}\left(t\right)\right)^{\left(1-\pa\right)q-1}}C_M\left(\frac{2-2\pa}{2-\pa}\right)^{-\alpha\frac{\beta^{\prime}-1}{\beta}}.
\end{equation}
so that 
\begin{equation}\label{Estimate_on_II-2}
II^{1/q}\leq C_MV\left(x,\mathcal{R}\left(t\right)\right)^{\pa/2}\left(\frac{2-2\pa}{2-\pa}\right)^{-\alpha\frac{\beta^{\prime}-1}{\beta}\left(1-\frac{\pa}{2}\right)}.
\end{equation}

Plugging the estimates of $I$ and $II$ from \eqref{Estimate_on_I} and \eqref{Estimate_on_II-2} into \eqref{eq-int-pd} and applying $\left(\VD\right)$ and \eqref{eq:R-regularity} 
we then obtain that 
\begin{align}
 & \int_{D}p_{D}\left(x,y,t\right)d\mu\left(y\right)\nonumber \le C_M\left(\frac{2-2\pa}{2-\pa}\right)^{-\alpha\frac{\beta^{\prime}-1}{\beta}\left(1-\frac{\pa}{2}\right)}\frac{V\left(x,\mathcal{R}\left(t\right)\right)^{\pa/2}}{V\left(x,\mathcal{R}\left(\delta t\right)\right)^{\pa/2}}e^{-\pa\left(1-\delta\right)t\lambda\left(D\right)}
\end{align}
By $\left(\VD\right)$ and \eqref{eq:R-regularity} we know that
\[
\frac{V\left(x,\mathcal{R}\left(t\right)\right)^{\pa/2}}{V\left(x,\mathcal{R}\left(\delta t\right)\right)^{\pa/2}}
\le
C_{\alpha}^{\frac{\pa}2}\left(\frac{\mathcal{R}\left(t\right)}{\mathcal{R}\left(\delta t\right)}\right)^{\alpha\frac{\pa}{2}}
\le C_{\alpha}^{\frac{\pa}2}C_{F}^{\frac{\alpha\pa}{2\beta}}\delta^{-\frac{\alpha}{\beta}\cdot\frac{\pa}{2}}.
\]
Hence
\begin{align}
 & \int_{D}p_{D}\left(x,y,t\right)d\mu\left(y\right)
 \leq C_{M,\pa,\delta}e^{-\pa\left(1-\delta\right)t\lambda\left(D\right)}\label{eq:final_est_on_both}
\end{align}
where
\[
C_{M,\pa,\delta}=C_M\left(\frac{2-2\pa}{2-\pa}\right)^{-\alpha\frac{\beta^{\prime}-1}{\beta}\left(1-\frac{\pa}{2}\right)}\delta^{-\frac{\alpha}{\beta}\cdot\frac{\pa}{2}},
\]
and $C_M$ is a constant independent of $\pa,\delta,t,D$. It is only dependent on constants in $\left(\VD\right)$, $\left(\UEF\right)$ and \eqref{eq:F-regularity}. This proves the Claim \eqref{eq-claim} under conditions $\left(\VD\right)$ and $\left(\UEF\right)$.


\textbf{\underline{Step 2:}} Next we show \eqref{eq-claim} holds for any domain  that has a positive bottom of the spectrum, i.e. $\lambda(D)>0$. For any domain $D\subset M$ that satisfies $\lambda\left(D\right)>0$, we can find open sets of finite measure $D_{n}\subset M$, $n=1,2,\dots$ such
that $D_{n}\subset D_{n+1}$ for all $n\geq1$ and $D=\cup_{n\geq1}D_{n}$.
Then by step 1 we know that for a.e. $x\in D$, 
\begin{align*}
P_{t}^{D}1_{D}(x) & =\lim_{n\to\infty}P_{t}^{D_{n}}1_{D_{n}}\left(x\right)\\
 & \leq\limsup_{n\to\infty}C_{\pa,\delta}e^{-\pa\left(1-\delta\right)t\lambda\left(D_{n}\right)}
\end{align*}
Note here $C_{\pa,\delta}$ does not depend on any of the domains. Then by the monotonicity of $\lambda(\cdot)$ and the fact that $\lim_{n\to\infty}\lambda\left(D_{n}\right)=\lambda\left(D\right)$ we obtain \eqref{eq-claim} for domains with positive bottom of the spectrum. \\

\textbf{\underline{Step 3:}} 
Next we show that there exists a properly exceptional set $\mathcal{N}\subset M$ and   a constant $C_{\pa,\delta}>0$ such that for every $\pa,\delta\in\left(0,1\right)$ 
\begin{equation}\label{eq-bound-Q}
\sup_{x\in D\backslash\mathcal{N}}\mathbb{P}_{x}\left(\tau_{D}>t\right)\leq  C_{\pa,\delta}\, e^{-\pa(1-\delta ) t\lambda(D)}
\end{equation}
for any $t>0$. First recall from Section \ref{Sec:Diff} that for $\mu-$a.e. $x\in D$, we have
\[
\P_x(\tau_D>t)=\E_x(1_D(X^D_t))=P_t^D 1_D(x).
\] 
{   See the discussion in \cite[Section 6.2, page 554]{Grigoryan-Hu-2014} for the comparison of of $\P_x(\tau_D>t)$ and $P_t^D 1_D(x)$. }

It is known that for any $t>0$ and any $\Omega\subset M$ satisfying $\mu\left(\Omega\right)<\infty$,  $\mathbb{P}_{x}\left(\tau_{\Omega}>t\right)$ as a function of $x$,
is a quasi-continuous realization of $P_{t}^{\Omega}1_{\Omega}(x)$ (see \cite[Theorem 4.2.3]{Fukushima-book-ed1-1994}).  Thus by Proposition \ref{GH-Prop6.1} and inequality \eqref{eq-claim}  there exists a properly exceptional set  $\mathcal{N}_{\pa,\delta,t,D}\subset M$ such that for all $x\in D\setminus \mathcal{N}_{\pa,\delta,t,D}$,
\begin{equation}\label{eq-invisible}
\P_x(\tau_D>t)\leq C_{\pa,\delta}\, e^{-\pa\left(1-\delta\right)\lambda\left(D\right)t}.
\end{equation}
We now use an argument inspired from \cite{Grigoryan-Hu-2014} by constructing a properly exceptional set. Let $E$ be a dense countable subset of $M$. Let $\mathcal{S}$ be the collection of balls  of rational radii centered at the points of $E$, i.e. 
\[
\mathcal{S}=\{B(x,r); x\in E, r\in \mathbb{Q}_+\}.
\] 
Let $\mathcal{T}$ be the collection of sets that are finite unions of balls in $\mathcal{S}$. For any set $U\in \mathcal{T}$, we can find an invisible set $\mathcal{N}_{\pa,\delta,t,U}$ as described above. Consider
\[
\mathcal{N}:=\bigcup_{U\in\mathcal{T},t\in \mathbb{Q}_+,\pa,\delta\in \mathbb{Q}_+\cap (0,1)  } \mathcal{N}_{\pa,\delta,t,U}
\]
Clearly $\mathcal{N}$ is also a properly exceptional set. We claim that it is the desired set. First for any given open $D\subset M$,  we can find
an increasing sequence $D_1\subset D_2\subset\cdots $ in $\mathcal{T}$ such that $D=\cup_{k=1}^\infty D_k$. Therefore for any $x\in D\setminus \mathcal{N}$, we know that there exists an $m$ such that for all $k\ge m$, it holds that $x\in D_k\setminus \mathcal{N}$ hence $x\in D_k\setminus \mathcal{N}_{\pa,\delta,t,D_k}$ for all $t\in \mathbb{Q}_+,\pa,\delta\in \mathbb{Q}_+\cap (0,1)$ and $k\ge m$.
By \eqref{eq-invisible} we then obtain that for all $k\ge m$ and $t\in \mathbb{Q}_+,\pa,\delta\in \mathbb{Q}_+\cap (0,1)$,
\[
\mathbb{P}_{x}\left(\tau_{D_k}>t\right)
= P_t^{D_k} 1_{D_k}(x)
\leq C_{\pa,\delta} e^{-\pa\left(1-\delta\right)t\lambda(D_k)}.
\]
Take $k\to\infty$ we obtain that 
\[
\lim_{k\to\infty}\mathbb{P}_{x}\left(\tau_{D_k}>t\right)\le \limsup_{k\to\infty}C_{\pa,\delta} e^{-\pa\left(1-\delta\right)t\lambda(D_k)}\le  C_{\pa,\delta} e^{-\pa\left(1-\delta\right)t\lambda(D)}.
\]
The last inequality comes from the fact that $D_{k}\subset D$ implies $\lambda\left(D_{k}\right)\geq\lambda\left(D\right)$.
Hence \eqref{eq-bound-Q} holds for any $t\in \mathbb{Q}_+,\pa,\delta\in \mathbb{Q}_+\cap (0,1)$. We then have that \eqref{eq-bound-Q} holds for all $t\in \R_+,\pa,\delta\in (0,1)$ using a dense argument.



\textbf{\underline{Step 4:}} Lastly, we obtain the sharp constants. From \eqref{eq:final_est_on_both} we have 
\[
\sup_{x\in D\setminus \mathcal{N}} \P_x(\tau_D>t) \leq  C_M\left(\frac{2-2\pa}{2-\pa}\right)^{-\frac{\beta^\prime-1}{\beta}\alpha\left(1-\frac{\pa}{2}\right)}\delta^{-\frac{\alpha}{\beta}\cdot\frac{\pa}{2}}e^{-\pa\left(1-\delta\right)t\lambda\left(D\right)}
\]
Let $1-\epsilon:=(1-\delta)\pa$, and $\mu=1-\pa$ then we have $\frac1\delta=\frac{\pa}{\pa-1+\epsilon}=\frac{1-\mu}{\epsilon-\mu}\le \frac{1}{\epsilon-\mu}$. Thus
\begin{align*}
\sup_{x\in D\setminus \mathcal{N}} \P_x(\tau_D>t) &\le  C_M\left(1+\frac1\mu\right)^{(\frac12+\frac{\mu}2)\frac{\beta^\prime-1}{\beta}\alpha} \left(\frac{1}{\epsilon-\mu} \right)^{\frac{\alpha}{\beta} \frac{1-\mu}{2}}e^{-(1-\epsilon ) t\lambda(D)}.
\end{align*}
Let $d':=\frac{\alpha}{\beta} \vee \frac{\beta^\prime-1}{\beta}\alpha$
and $\mu=\frac12\epsilon$, then the above inequality becomes
\begin{align}\label{eq-est-vogt}
\sup_{x\in D\setminus \mathcal{N}} \P_x(\tau_D>t) & \le C_M \left(1+\frac2\epsilon\right)^{(\frac12+\frac{\mu}2)d'} \left(\frac{2}{\epsilon} \right)^{d'(\frac{1-\mu}2)}e^{-(1-\epsilon ) t\lambda(D)}\notag\\
& \le  C_M\left(1+\frac2\epsilon\right)^{d'}e^{-(1-\epsilon ) t\lambda(D)}.
\end{align}
This completes the proof of Theorem \ref{thm:Prob-Estimate}.

\subsection{Proof of Corollary \ref{Cor:polynomial-bound}}\label{sec-proof-cor}

Let  $a:=t\lambda(D)/d'$ and
\[
E(\epsilon,a):=\left(1+\frac2\epsilon\right)e^{\epsilon a}.
\]
Then by \eqref{eq-est-vogt} we know that for any $\epsilon\in(0,1)$,
\[
\sup_{x\in D\setminus \mathcal{N}} \P_x(\tau_D>t) \le C E(\epsilon, a)^{d'}e^{-\lambda(D)t}.
\]
We then choose $\epsilon$ according to the following situations.
\begin{itemize}
\item[(1)] When $a\ge1$, by taking $\epsilon=\frac{1}{a}$ we have
\[
E(\epsilon, a)= \left(1+2a\right)e.
\]

\item[(2)] When $a<1$, by taking $e^{\epsilon a}\le 1+\frac{e^\epsilon-1}{\epsilon} a$ we obtain
\[
E(\epsilon, a)\le \left(1+\frac2\epsilon\right) \left(1+\frac{e^\epsilon-1}{\epsilon} a\right).
\]
Letting $\epsilon\to 1$ we have
\[
E(\epsilon, a)\le 9 (1+2 a).
\]
\end{itemize}
At the end we obtain that
\[
\sup_{x\in D\setminus \mathcal{N}} \P_x(\tau_D>t)  \le K \left(1+\frac{2\lambda(D)}{d'} \,t\right)^{d'}e^{-\lambda(D)t}.
\]
where $K>0$ is a constant that depends only on the geometric and analytical constants of $M$,  namely, dimension constants $\alpha$ and $\beta,\beta^\prime$, volume doubling constants $C_\alpha$, and sub-Gaussian bound constants $\cue$ and $\Cue$.

\section{Mean Exit Time Estimates and Sharp Lower Bound for the Survival Probability}\label{sec:Mean-Lower}


\subsection{Mean Exit Time Estimates}

As a consequence of Theorem \ref{thm:Prob-Estimate} we obtain spectral bounds for the mean exit time on metric measure spaces in the proposition below. This in turn gives a characterization for domains that have a positive bottom of the spectrum. The classical results of the type \eqref{Torsion} and  \eqref{eq:Intro-4} have been known in $\mathbb{R}^n$ in \cite{Banuelos-Carrol1994,Vandenberg-Carroll-2009,Giorgi-Smits-2010}.  Let $\esup_{x\in A} f$ be the essential supremum of a function $f$ on a set $A$ with respect to $\mu$.

\begin{prop}\label{Cor:main_product_bound} Under the same assumption as Theorem \ref{thm:Prob-Estimate}, we have

(1) For any domain $D\subset M$,
\begin{equation}
\lambda\left(D\right)>0\text{ if and only if }\esup_{x\in D}\mathbb{E}_x\left[\tau_{D}\right]<\infty.\label{eq:Thm:iff1}
\end{equation}

(2) For any $p>0,$ there exists a constant $C_{p}\geq\Gamma\left(p+1\right)$
such that for every domain $D\subset M$ satisfying $\lambda\left(D\right)>0$,
we have
\begin{equation}
\Gamma\left(p+1\right)\leq\lambda\left(D\right)^{p}\cdot\esup_{x\in D}\mathbb{E}_{x}\left[\tau_{D}^{p}\right]\leq C_{p},\label{eq:Thm:iff2}
\end{equation}
where $C_{p}$ is a constant depending only on $p$ and $\alpha,\alpha^{\prime},\beta$
and the constants in the assumptions of Theorem \ref{thm:Prob-Estimate}. The upper bound
holds everywhere outside a properly exceptional set. 

(3) For any domain $D\subset M$, if $\esup_{x\in D}\mathbb{E}_{x}\left[e^{a\tau_{D}}\right]<\infty$ {for some $a>0$}
then
\[
\lambda\left(D\right)>a+\frac{a}{\esup_{x\in D}\mathbb{E}_{x}\left[e^{a\tau_{D}}\right]-1}.
\]

\end{prop}

\begin{remark}\label{rk:Lower-Bound}
The lower bound in \eqref{eq:Thm:iff2} for the $p=1$ case has already been known in the literature; If $\esup_{x\in D}\mathbb{E}_{x}\left[\tau_{D}\right]<\infty$ then
\begin{equation}\label{eq:Lower-Bound}
\lambda(D)\esup_{x\in D}\mathbb{E}_{x}\left[\tau_{D}\right]\geq1, \quad \text{for any domain } D\subset M.
\end{equation}
See \cite{Banuelos-Carrol1994,Giorgi-Smits-2010} in $\R^n$, see \cite{Telcs-1995,Telcs-2001} for the setting of graphs and \cite{Grigoryan-Hu-2014} for the setting of metric measure spaces. In fact, the above holds without the assumptions of $\left(\VD\right)$ or
$\left(\UEF\right)$.

\end{remark}

\begin{remark}\label{rem:positive-spectrum-ex}

(1) Proposition \ref{Cor:main_product_bound} holds for any bounded set $D\subset M$ due to the fact that $\esup_{x\in D}\mathbb{E}_{x}\left[\tau_{D}\right]<\infty$. 

(2) One can also obtain the results of Proposition \ref{Cor:main_product_bound} for any set $D\subset M$ satisfying $\mu\left(D\right)<\infty$ with a reverse volume doubling assumption. This is because a  Faber-Krahn inequality holds for this situation (see Proposition \ref{MainThm2}) which imply that $\lambda\left(D\right)>0$.
\end{remark}


To prove Proposition \ref{Cor:main_product_bound} we will need two lemmas which we present below. In particular, Lemma \ref{Lem-Balance-Lowerbound} is a generalization of Remark \ref{rk:Lower-Bound} which may also be of independent interest. Its proof follows arguments found in \cite{Grigoryan-Hu-2014,Grigoryan-Telcs-2012}. Other arguments can be found in \cite{Banuelos-Carrol1994,Banuelos-Mariano-Wang-2020,Giorgi-Smits-2010} for other settings.

\begin{lemma}\label{Lem-Balance-Lowerbound}
Suppose that $g:\left[0,\infty\right)\to\left[0,\infty\right)$ and
$g\in L^{1}\left(\left[0,T\right)\right)$ for all $T\geq0$. If $\esup_{x\in D}\mathbb{E}_{x}\left[\int_{0}^{\tau_{D}}g\left(t\right)dt\right]<\infty$
then 
\begin{equation}
\int_{0}^{\infty}g\left(t\right)e^{-t\lambda\left(D\right)}dt\leq\esup_{x\in D}\mathbb{E}_{x}\left[\int_{0}^{\tau_{D}}g\left(t\right)dt\right].\label{eq:Lem-Balance-Lowerbound}
\end{equation}
\end{lemma}

\begin{proof}
Recall that $\mathcal{L}_{D}$ is the generator of the Dirichlet form
$\left(\mathcal{E},\mathcal{F}\left(D\right)\right)$ in $L^{2}\left(D,\mu\right)$.
For any $T>0,$ we define the operator 
\[
G_{T,g}=\int_{0}^{T}g\left(t\right)e^{-t\mathcal{L}_{D}}dt=\varphi_{T,g}\left(\mathcal{L}_{D}\right)
\]
where 
\[
\varphi_{T,g}\left(\lambda\right) =\int_{0}^{T}g\left(t\right)e^{-t\lambda}dt.
\]
Note that by the assumptions on $g$ we have that $\varphi_{T,g}$
is bounded and continuous on $\left[0,\infty\right)$. Thus $G_{T,p}$
is a bounded self-adjoint operator in $L^{2}\left(D,\mu\right)$.
Moreover, note that $\varphi_{T,g}$ is decreasing so that by the
spectral mapping theorem we have
\begin{align}
\varphi_{T,g}\left(\lambda\left(D\right)\right) & =\varphi_{T,g}\left(\inf\text{spec}\left(\mathcal{L}_{D}\right)\right)  =\sup\text{spec}\left(G_{T,g}\right)  =\left\Vert G_{T,g}\right\Vert _{2\to2}.\label{eq:Lem-Balance-0}
\end{align}
We now claim that 
\begin{equation}
\left\Vert G_{T,g}\right\Vert _{2\to2}\leq\esup_{x\in D}\mathbb{E}_{x}\left[\int_{0}^{\tau_{D}}g\left(t\right)dt\right]<\infty.\label{eq:Lem-Balance-1}
\end{equation}

The operator $P_{t}^{D}=e^{-t\mathcal{L}_{D}}$ extends to an operator
in $L^{\infty}\left(D\right)$, and hence so does $G_{T,g}$. We also
use the fact that $P_{t}^{D}$ and $\mathcal{P}_{t}^{D}$ coincide
as operators in $L^{\infty}\left(D\right)$. Thus for any bounded
Borel function $f$, we have the following holds $\mu-$ almost
everywhere,
\[
G_{T,g}f=\int_{0}^{T}g\left(t\right)\left(\mathcal{P}_{t}^{D}f\right)dt.
\]
For a.e. $x\in D$ we have that 
\begin{align*}
\left|G_{T,g}f\left(x\right)\right| & \leq\int_{0}^{\infty}g\left(t\right)\left(\mathcal{P}_{t}^{D}\left|f\right|\right)\left(x\right)dt =\int_{0}^{\infty}g\left(t\right)\mathbb{E}_{x}\left[\left|f\left(X_{t}\right)\right|1_{\left(\tau_{D}>t\right)}\right]dt\\
 & =\mathbb{E}_{x}\left[\int_{0}^{\tau_{D}}g\left(t\right)\left|f\left(X_{t}\right)\right|dt\right]
 \leq\mathbb{E}_{x}\left[\int_{0}^{\tau_{D}}g\left(t\right)dt\right]\sup_{x\in D}\left|f\right|\\
 & \leq\esup_{x\in D}\mathbb{E}_{x}\left[\int_{0}^{\tau_{D}}g\left(t\right)dt\right]\sup_{x\in D}\left|f\right|.
\end{align*}
Hence 
\begin{equation}
\left\Vert G_{T,g}\right\Vert _{\infty\to\infty}\leq\esup_{x\in D}\mathbb{E}_{x}\left[\int_{0}^{\tau_{D}}g\left(t\right)dt\right].\label{eq:Lem-Balance-2}
\end{equation}
This implies that for any $h\in L^{1}\left(D,\mu\right)\cap L^{2}\left(D,\mu\right)$
we have
\begin{align*}
\left\Vert G_{T,g}h\right\Vert _{1} & =\sup_{f\in C_{0}\left(D\right)\setminus\left\{ 0\right\} }\frac{\int\left(G_{T}h\right)fd\mu}{\left\Vert f\right\Vert _{\infty}} =\sup_{f\in C_{0}\left(D\right)\setminus\left\{ 0\right\} }\frac{\int h\left(G_{T}f\right)d\mu}{\left\Vert f\right\Vert _{\infty}}\\
 & \leq\esup_{x\in D}\mathbb{E}_{x}\left[\int_{0}^{\tau_{D}}g\left(t\right)dt\right]\left\Vert h\right\Vert _{1}
\end{align*}
so that 
\begin{equation}
\left\Vert G_{T,g}\right\Vert _{1\to1}\leq\esup_{x\in D}\mathbb{E}_{x}\left[\int_{0}^{\tau_{D}}g\left(t\right)dt\right].\label{eq:Lem-Balance-3}
\end{equation}
Using (\ref{eq:Lem-Balance-2}) and (\ref{eq:Lem-Balance-3}) with
the Riesz-Thorin interpolation theorem we have
\begin{align}
\left\Vert G_{T,g}\right\Vert _{2\to2} & \leq\left\Vert G_{T,g}\right\Vert _{1\to1}^{\frac{1}{2}}\left\Vert G_{T,g}\right\Vert _{\infty\to\infty}^{\frac{1}{2}}\nonumber \\
 & \leq\esup_{x\in D}\mathbb{E}_{x}\left[\int_{0}^{\tau_{D}}g\left(t\right)dt\right].\label{eq:Lem-Balance-4}
\end{align}
Thus putting $(\ref{eq:Lem-Balance-4})$ in (\ref{eq:Lem-Balance-0})
we have 
\[
\int_{0}^{T}g\left(t\right)e^{-t\lambda\left(D\right)}dt\leq\esup_{x\in D}\mathbb{E}_{x}\left[\int_{0}^{\tau_{D}}g\left(t\right)dt\right],
\]
 and taking $T\to\infty$ gives us the desired bound (\ref{eq:Lem-Balance-Lowerbound}).
\end{proof}

\begin{lemma}
Let $\epsilon\in\left(0,1\right)$ and $g:\left[0,\infty\right)\to\left[0,\infty\right)$. Assume the same assumptions as
Theorem \ref{thm:Prob-Estimate}. Then 
there exists a properly exceptional set $\mathcal{N}$ and a constant
$C_{\epsilon}\geq1$ such that 
\begin{equation}
\sup_{x\in D\backslash\mathcal{N}}\mathbb{E}_{x}\left[\int_{0}^{\tau_{D}}g\left(t\right)dt\right]\leq C_{\epsilon}\int_{0}^{\infty}g\left(t\right)e^{-(1-\epsilon)\lambda(D)t}dt.\label{eq:Lem-Balance-Upperbound}
\end{equation}
\end{lemma}

\begin{proof}
This follows directly from Theorem \ref{thm:Prob-Estimate}, as for all $x\in D\backslash\mathcal{N}$
one has
\begin{align*}
&\mathbb{E}_{x}\left[\int_{0}^{\tau_{D}}g\left(t\right)dt\right]=  \mathbb{E}_{x}\left[\int_{0}^{\infty}g\left(t\right)1_{\left(\tau_{D}>t\right)}dt\right]\\
 &\quad\quad = \int_{0}^{\infty}g\left(t\right)\mathbb{P}_{x}\left(\tau_{D}>t\right)dt\leq  C_{\epsilon}\int_{0}^{\infty}g\left(t\right)e^{-(1-\epsilon)\lambda(D)t}dt,
\end{align*}
as desired.
\end{proof}

We are now ready to prove Proposition \ref{Cor:main_product_bound}. 

\begin{proof}[\textbf{Proof of Proposition \ref{Cor:main_product_bound}}]
Let $p>0$. Take $g(t)=pt^{p-1}.$ Then applying \eqref{eq:Lem-Balance-Lowerbound} from Lemma \ref{Lem-Balance-Lowerbound}
one has that if $\esup_{x\in D}\mathbb{E}_{x}\left[\int_{0}^{\tau_{D}}g\left(t\right)dt\right]=\esup_{x\in D}\mathbb{E}_{x}\left[\tau_{D}^{p}\right]<\infty$
then 
\[
\frac{\Gamma\left(p+1\right)}{\lambda\left(D\right)^{p}}=\int_{0}^{\infty}g\left(t\right)e^{-t\lambda\left(D\right)}dt\leq\esup_{x\in D}\mathbb{E}_{x}\left[\tau_{D}^{p}\right].
\]
If $\lambda\left(D\right)>0$, by applying (\ref{eq:Lem-Balance-Upperbound})
we have that 
\begin{align*}
\sup_{x\in D\backslash\mathcal{N}}\mathbb{E}_{x}\left[\int_{0}^{\tau_{D}}g\left(t\right)dt\right] & \leq C_{\epsilon}\int_{0}^{\infty}pt^{p-1}e^{-(1-\epsilon)\lambda(D)t}dt =\frac{C_{\epsilon}}{\left(1-\epsilon\right)^{p}}\frac{\Gamma\left(p+1\right)}{\lambda\left(D\right)^{p}}.
\end{align*}
This proves (\ref{eq:Thm:iff1}) and (\ref{eq:Thm:iff2}). 

It is left to prove part (3).  Suppose $\esup_{x\in D}\mathbb{E}_x\left[e^{a\tau_{D}}\right]<\infty$.
Thus for almost every $x\in D$
\[
\mathbb{E}_{x}\left[\int_{0}^{\tau_{D}}e^{at}dt\right]=\frac{1}{a}\mathbb{E}_{x}\left[e^{a\tau_{D}}-1\right]<\infty.
\]
Using (\ref{eq:Lem-Balance-Lowerbound}) with $g(t)=e^{at}$ we have
that 
\[
\frac{1}{\lambda\left(D\right)-a}=\int_{0}^{\infty}e^{at}e^{-t\lambda\left(D\right)}dt\leq\esup_{x\in D}\frac{1}{a}\mathbb{E}_{x}\left[e^{a\tau_{D}}-1\right]<\infty,
\]
Rearranging this inequality gives the desired result. 
\end{proof}

\subsection{Sharp Lower Bound}

We now give a result on the lower bound of the survival probability which matches with the rate of decay of the upper bound obtained in Theorem \ref{thm:Prob-Estimate}. This lower bound is general and does not assume sub-Gaussian bounds or volume doubling. 

\begin{prop}[Lower Bound]\label{Spectral_Lower_Bound}
Assume the metric measure space $\left(M,d,\mu\right)$ is endowed
with a regular, (not necessarily local) Dirichlet form $\left(\mathcal{E},\mathcal{F}\right)$.
Then 
\begin{equation}\label{eq:Lower_bound_pf1}
    e^{-\lambda\left(D\right)t}\leq\esup_{x\in D}\mathbb{P}_x\left(\tau_{D}>t\right).
\end{equation}

\end{prop}

\begin{proof}

{    

Let $a>0$. Suppose first that $\lambda(D)>0$. First recall that 
\begin{equation*}
\mathbb{E}_{x}\left[\int_{0}^{\tau_{D}}g\left(t\right)dt\right]=  \mathbb{E}_{x}\left[\int_{0}^{\infty}g\left(t\right)1_{\left(\tau_{D}>t\right)}dt\right]= \int_{0}^{\infty}g\left(t\right)\mathbb{P}_{x}\left(\tau_{D}>t\right)dt,
\end{equation*}
so that by Lemma \ref{Lem-Balance-Lowerbound} we have
\begin{equation}
\int_{0}^{\infty}g\left(t\right)e^{-t\lambda\left(D\right)}dt\leq\esup_{x\in D}\int_{0}^{\infty}g\left(t\right)\mathbb{P}_{x}\left(\tau_{D}>t\right)dt \leq \int_{0}^{\infty}g\left(t\right) \esup_{x\in D}\mathbb{P}_{x}\left(\tau_{D}>t\right)dt.\label{eq:Lem-Balance-Lowerbound-2}
\end{equation}
Applying \eqref{eq:Lem-Balance-Lowerbound-2} to  $g_{n}(t)=n1_{\left[a,a+\frac{1}{n}\right]}\left(t\right)$, which is allowed since $$\esup_{x\in D}\int_{0}^{\infty}n1_{\left[a,a+\frac{1}{n}\right]}\left(t\right)\mathbb{P}_{x}\left(\tau_{D}>t\right)dt\leq 1$$
and letting 
\[
f\left(t\right)=\esup_{x\in D}\mathbb{P}_{x}\left(\tau_{D}>t\right)
\]
we have 
\begin{equation}
n\int_{a}^{a+\frac{1}{n}}e^{-t\lambda\left(D\right)}dt\leq n\int_{a}^{a+\frac{1}{n}}f\left(t\right)dt.\label{eq:LebDiff-1}
\end{equation}
If $f$ is right-continuous then by letting $n\to\infty$
in $\left(\ref{eq:LebDiff-1}\right)$ gives us 
\[
e^{-a\lambda\left(D\right)}\leq\esup_{x\in D}\mathbb{P}_{x}\left(\tau_{D}>a\right),
\]
as needed. We devote the rest of the proof to show the right-continuity of $f$. 

First we claim $f$ is non-increasing. To do this, note that for a.e. $x\in D$, the map $t\mapsto\mathbb{P}_{x}\left(\tau_{D}>t\right)$
is non-increasing. For any $a<b$ and
a.e. $x\in D$ we have 
\[
\mathbb{P}_{x}\left(\tau_{D}>b\right)\leq\mathbb{P}_{x}\left(\tau_{D}>a\right)\leq\esup_{x\in D}\mathbb{P}_{x}\left(\tau_{D}>a\right)
\]
hence 
\[
f(b)=\esup_{x\in D}\mathbb{P}_{x}\left(\tau_{D}>b\right)\leq\esup_{x\in D}\mathbb{P}_{x}\left(\tau_{D}>a\right)=f(a)
\]
showing that $f$ is non-increasing as well. 

Now to see right-continuity of $f$, we take any sequence $t_{m}\downarrow a$ and show that $\lim_{m\to\infty}f(t_m)=f(a)$. Since
$f$ is non-increasing and $t_{m}$ decreases down to $a$ then 
\[
\lim_{m\to\infty}f\left(t_{m}\right)=\sup_{m:m\in\mathbb{N}}f\left(t_{m}\right)=\sup_{m:m\in\mathbb{N}}\esup_{x\in D}\mathbb{P}_{x}\left(\tau_{D}>t_{m}\right).
\]

By interchanging the order of supremum and the essential supremum, and using the fact that  $t\mapsto\mathbb{P}_{x}\left(\tau_{D}>t\right)$
is right-continuous which implies that $\sup_{m:m\in\mathbb{N}}\mathbb{P}_{x}\left(\tau_{D}>t_{m}\right)=\lim_{m\to\infty}\mathbb{P}_{x}\left(\tau_{D}>t_{m}\right)=\mathbb{P}_{x}\left(\tau_{D}>a\right)$, we have 
\[
\lim_{m\to\infty}f\left(t_{m}\right)= \esup_{x\in D}\sup_{m:m\in\mathbb{N}}\mathbb{P}_{x}\left(\tau_{D}>t_{m}\right) =\esup_{x\in D}\mathbb{P}_{x}\left(\tau_{D}>a\right)=f(a).
\]

}

If $\lambda (D)=0$, then we can still apply \eqref{eq:Lem-Balance-Lowerbound}, but the left-hand-side of \eqref{eq:Lower_bound_pf1} would be equal to 1, and the rest of the proof follows. 
\end{proof}

\section{Exit time asymptotics}\label{sec:smallball}

The large time asymptotic estimates for $\mathbb{P}_{x}\left(\tau_{D}>t\right)$ has been well studied in the Euclidean spaces. It is a classical result in $\R^n$ that 
\[
\lim_{t\to\infty}\frac{\log\mathbb{P}_{x}\left(\tau_{D}>t\right)}{t}=-\lambda\left(D\right)
\]
holds for all domains $D$. There are also nontrivial results concerning the case 
when $\lambda\left(D\right)=0$ (\cite{Banuelos-DeBlassie-Smits-2001,Wenbo-Li-2003,Lifshits-Shi-2002}). Numerous works have been devoted to studying this type of asymptotic for other stochastic
processes as well, such as the Doob $h$-conditional diffusion (see  \cite{DeBlassie-1988}), iterated Brownian
motion (see \cite{Nane-2007}), integrated Brownion motions (see (see \cite{Chen-Li-2003,Khoshnevisan-1998}),symmetric-stable process (see \cite{Mendez-Hernandez-2002}), and horizontal Brownian motion on Carnot groups (see \cite{Carfagnini-Gordina-2020,Carfagnini-Gordina_Carnot-2022}).

We  obtain such asymptotic estimates for $\mathbb{P}_{x}\left(\tau_{D}>t\right)$ in a general metric measure space setting. Before stating the result, we need the following definitions.

\begin{df}\label{def:irreducible}
We say a Borel set $A\subset M$ is \emph{$P_{t}-$ invariant }if
$P_{t}\left(1_{A}f\right)=1_{A}P_{t}f$, $\mu-$a.e. for any bounded function $f\geq0$
and $t>0$. 

We say a Dirichlet form $\left(\mathcal{E},\mathcal{F}\right)$
is \emph{irreducible} if any $P_{t}-$ invariant set satisfies either
$\mu\left(A\right)=0$ or $\mu\left(A^{c}\right)=0$. 

{    For any set $D\subset M$, we define the \emph{capacity} by $\text{Cap}\left(E\right)=\inf_{\varphi}\mathcal{E}_{1}\left(\varphi\right)$
where the infimum is taken over all $\varphi\in\mathcal{F}$ such
that $\varphi\geq1$ in an open neighborhood of $E$. }

\end{df}

It turns out that an asymptotic upper bound comes directly from Theorem \ref{thm:Prob-Estimate}, while the asymptotic lower bound holds under the irreducibility assumption.

\begin{prop}
\label{thm:Main_Asymptotic}
Under the same assumption as Theorem \ref{thm:Prob-Estimate}, there exists a properly exceptional set $\mathcal{N}\subset M$
such that for every domain $D\subset M$ we have 
\begin{equation}\label{eq:Asymp1'}
\limsup_{t\to\infty}\frac{\log\mathbb{P}_{x}\left(\tau_{D}>t\right)}{t}\leq-\lambda\left(D\right)
\end{equation}
for every $x\in D\backslash\mathcal{N}$. 

If in addition we assume the Dirichlet form $\left(\mathcal{E},\mathcal{F}(D)\right) $ is irreducible, then 
\begin{equation}
\lim_{t\to\infty}\frac{\log\mathbb{P}_{x}\left(\tau_{D}>t\right)}{t}=-\lambda\left(D\right)\label{eq:Asymp1} 
\end{equation}
for every  $x\in D\backslash\mathcal{N}$. 

\end{prop}
\begin{proof}
Suppose $\lambda (D)>0$. The upper bound \eqref{eq:Asymp1'} is an easy consequence of Theorem \ref{thm:Prob-Estimate}.

Now let us further assume that $\left(\mathcal{E},\mathcal{F}(D)\right) $
is irreducible. Then by \cite[Theorem 4.1]{Fukushima-Takeda-1984}, it has been shown that 
\begin{equation}
\liminf_{t\to\infty}\frac{\log\mathbb{P}_{x}\left(\tau_{D}>t\right)}{t}\geq-\lambda\left(D\right)\label{eq:asmp_proof_lower_1}
\end{equation}
holds quasi-everywhere, that is, it holds outside some set $E$ of capacity
$0$. In fact, this is true without having to assume $\left(\VD\right)$ or
$\left(\UE\right)$. To obtain
the lower bound in (\ref{eq:Asymp1}), it is sufficient to
assume only that the Dirichlet form is regular and irreducible. 


Now, using a similar argument as in the proof of \cite[Prop. 6.1]{Grigoryan-Hu-2014}, we know $E$ is contained inside some properly exceptional
set, hence (\ref{eq:asmp_proof_lower_1}) holds outside some properly
exceptional set. This give equality in \eqref{eq:Asymp1}.

If $\lambda (D)=0$, then \eqref{eq:asmp_proof_lower_1} is still true, hence $\liminf_{t\to\infty} \log\mathbb{P}_{x}\left(\tau_{D}>t\right)/t \geq  0$. Since $\mathbb{P}_{x}\left(\tau_{D}>t\right)\leq 1$ then $\limsup_{t\to\infty} \log\mathbb{P}_{x}\left(\tau_{D}>t\right)/t \leq 0$, giving the desired result in \eqref{eq:Asymp1}. 
\end{proof}

We point out that if the domain $D$ has a discrete spectrum, then \eqref{eq:Asymp1} follows immediately from the eigenfunction expansion. 

\begin{remark} {
Let $a>0$. By the Chernoff bound, we have the following for a random variable $X$: 
\begin{equation}\label{eq:exponential_int}
\mathbb{E}\left[e^{aX}\right]<\infty \quad \text{implies}  \quad  \liminf_{t\to\infty}-\frac{\log\mathbb{P}\left(X>t\right)}{t}>a. 
\end{equation}
Therefore combining with \eqref{eq:Asymp1}  we have  that for any domain $D\subset M$ and $x\in D$, 
\begin{equation}
\mathbb{E}_x\left[e^{a\tau_{D}}\right]<\infty \quad \text{ implies }\quad 
 \lambda\left(D\right)>a   .\label{eq:Thm:iff3}
\end{equation}
However, Proposition \ref{Cor:main_product_bound}-(3) provides a sharper bound than \eqref{eq:Thm:iff3}.
}

\end{remark}

Below we state a sufficient condition for irreducibility which is easier to work with in some cases.

\begin{prop}\label{prop:Irreducibility_cond}
Assume the metric space $\left(M,d,\mu\right)$ is endowed with a
regular Dirchlet form $\left(\mathcal{E},\mathcal{F}\right)$. If
the heat kernel $p\left(x,y,t\right)$ exists
and satisfies $p\left(x,y,t\right)>0$ for almost every  $\left(x,y\right)\in M\times M,t>0$,
then the Dirichlet form $\left(\mathcal{E},\mathcal{F}\right)$ is
irreducible. 
\end{prop}

\begin{proof}
Take a set $A\subset M$
to be $P_{t}$-invariant, namely  $P_{t}\left(1_{A}f\right)=1_{A}P_{t}f$ holds for any bounded $f\geq0$. We aim to show that either $\mu\left(A\right)=0$
or $\mu\left(A^{c}\right)=0$. 

Take $f=1_{B}$ for a Borel set $B\subset M$.
Then 
\begin{align}\label{eq:irre_pf_1}
\mathbb{P}_{x}\left(X_{t}\in A\cap B\right)=P_{t}\left(1_{A}1_{B}\right)\left(x\right)  =
1_{A}\left(P_{t}1_{B}\right)\left(x\right)   =1_{A}\left(x\right)\mathbb{P}_{x}\left(X_{t}\in 
B\right).
\end{align}
We then consider cases:

Case 1: Suppose $x\in A$. Take $B=A^{c}$ in (\ref{eq:irre_pf_1})
so that 
\[
0=\mathbb{P}_{x}\left(X_{t}\in\emptyset\right)=\mathbb{P}_{x}\left(X_{t}\in A^{c}\right)=\int_{A^{c}}p\left(x,y,t\right)d\mu(y).
\]
Since $p>0$ a.e. this must mean that 
$
\mu\left(A^{c}\right)=0.
$

Case 2: Suppose $x\notin A$. Take $B=A$ in (\ref{eq:irre_pf_1})
so that 
\[
\int_{A}p\left(x,y,t\right)d\mu(y)=\mathbb{P}_{x}\left(X_{t}\in A\right)=0,
\]
Similarly, since
$p>0$ a.e. this must mean that 
$
\mu\left(A\right)=0.
$
\end{proof}




\section{Application to the Hot-spots constant}\label{sec:HotSpots}

The \textbf{Hot-Spots conjecture} says that for a Lipschitz planar domain $D$ in $\R^d$, the maximum and minimum of the second Neumann eigenfunction  $\varphi_{2}$ are  attained
 on the boundary of $D$.  It has been
proven true for a number of classes of domains in $\mathbb{R}^{d}$
\cite{lip_domains,Banuelos-Burdzy1999,Jerison,triangles,thin_tubes,Pascu,Rohleder-2021}, but false for some other domains  (see \cite{Burdzy-Werner1999,Kleefeld}). Hence the conjecture
fails for arbitrary planar domains. It is widely believed that the conjecture should hold true for all convex domains, but this is still open.

A recent result of Kleefeld in \cite{Kleefeld} finds a domain $D\subset\R^d$ in which the Hot Spots conjecture fails in the following quantitative way:
\[
\frac{\sup_{x\in D}\varphi_{2}\left(x\right)}{\sup_{x\in\partial D}\varphi_{2}\left(x\right)}>{1+10^{-3}}.
\]
In  \cite{Steinerberger-2021a}, Steinerberger showed  that, however,  the Hot-Spots conjecture can not 
fail by an arbitrary factor. In fact for any bounded smooth domain $D\subset\mathbb{R}^{d}$ with
smooth boundary, Steinerberger shows that
\begin{equation}\label{HotSpotIneq}
\sup_{x\in D}\varphi_{2}\left(x\right)\leq60\sup_{x\in\partial D}\varphi_{2}\left(x\right).
\end{equation}
Once an inequality of the form of \eqref{HotSpotIneq} is obtained, then the next natural question is on finding the best possible constant in the inequality. 
We call this best constant for the above inequality to hold the \emph{Hot Spots constant}. Recently the bound on the Hot Spots constant has been improved from $60$ to $5.11$ for bounded Lipschitz planar domains in \cite{Mariano-Panzo-Wang-2021}. The bound can also be improved to $\sqrt{e}\approx 1.6487$ asymptotically for large dimension. Clearly, if the Hot Spots Constant can be proved to be $1$ for certain classes of domains then the weak Hot Spots conjecture would be proven. 

One can also consider the Hot-spots conjecture in Riemannian settings such as the work of \cite{thin_tubes} for tubular neighborhoods of curves on an arbitrary two-dimensional Riemannian manifold. Let $M$ be a $d$--
dimensional complete Riemannian manifold with
non-negative Ricci curvature lower bound. Consider a bounded domain $D\subset M$
with smooth boundary. Let $\Delta$ be the Laplace-Beltrami operator on $M$ and $\varphi_{2}$ 
the first non-trivial eigenfunction for $-\Delta$ under
Neumann boundary conditions, i.e.,
\begin{align*}
-\Delta\varphi_{2} & =\mu_{2}\varphi_{2}\text{ in }D\\
\frac{\partial\varphi_2}{\partial\nu} & =0 \quad\ \ \text{ on }\partial D
\end{align*}
where $\mu_{2}>0$ is the associated eigenvalue. We let $\lambda_1(D)$ be the first Dirichlet eigenvalue.

 { 
As an application of Theorem \ref{thm:Prob-Estimate}, we obtain an upper bound for the Hot Spots constant in a Riemannian setting.  
Unlike the results of \cite{Mariano-Panzo-Wang-2021,Steinerberger-2021a}, we do not have an explicit constant on the Hot Spots constant as it depends on the constants from the assumptions of Theorem \ref{thm:Prob-Estimate}. Given a constant $\hs\in(0,1)$, we will consider a class of domains $\mathcal{D}_\hs$ such that
\begin{equation}\label{eq:Dirch-Neuman-constant}
\mu_{2}\left(D\right)\leq \hs \lambda_{1}(D),
\end{equation}
holds for all  $D\in \mathcal{D}_\hs$. The constant $\hs$ is known as the \emph{ratio upper bound} for the class $\mathcal{D}_v$. Such classes of domains are natural to consider such as in \cite[Definition 3]{Mariano-Panzo-Wang-2021}. The existence of such a constant $\hs<1$ in \eqref{eq:Dirch-Neuman-constant} for classes of domains on manifolds is itself an interesting question that we do not address here. 
In the Euclidean case, inequality \eqref{eq:Dirch-Neuman-constant} holds for any bounded Lipschitz domain with a  constant 
$\hs_{\mathbb{R}^d}=\frac{\mu_{2}\left(B\right)}{\lambda_{1}\left(B\right)}<1$
where $B$ is a Euclidean ball. This follows readily by combining the classical Faber-Krahn and
Szego-Weinberger \cite{Weinberger-1956,Szego-1954} inequalities (see also \cite[Lemma 1]{Mariano-Panzo-Wang-2021} for a proof showing $\hs_{\mathbb{R}^d}<1$). 
}

\begin{prop}
\label{thm:Hot-Spots}Let $M$ be a Riemannian manifold with $\mbox{Ric}\geq 0$ and fix $\hs\in\left (0,1\right)$. Let $\mathcal{D}_\hs$ be a class of bounded smooth domains $D\subset M$ where the Neumann heat kernel $K_D(x,y,t)$ can be extended to a continuous function on $\overline{D}\times \overline{D}\times (0,\infty)$, and
\[
\mu_{2}\left(D\right)\leq \hs\lambda_{1}(D),
\]
for all $D\in \mathcal{D}_\hs$.
Then there exists a universal constant $C_{\mathcal{D}_\hs}\geq1$ independent of domains in $\mathcal{D}_\hs$
such that 
\begin{equation}\label{eq:Thm-HS-1}
\sup_{x\in D}\varphi_{2}\left(x\right)\leq C_{\mathcal{D}_\hs}\sup_{x\in\partial D}\varphi_{2}\left(x\right).
\end{equation}
\end{prop}

\begin{proof}
The proof follows from the same arguments made in \cite{Mariano-Panzo-Wang-2021,Steinerberger-2021a} together with Theorem \ref{thm:Prob-Estimate}. We sketch the proof here for completeness. 

The approach is probabilistic and relies on the existence of the reflected Brownian motion on Riemannian manifolds, which is in fact well-known. We refer to \cite[Chapter V.]{Ikeda-Watanabe-1989} for the theory of reflected Brownian motion on Riemannian manifolds. 

First we note that it suffices to consider domains where the
Hot-Spots conjecture does not hold, that is $\sup_{x\in\partial D}\varphi_{2}\left(x\right)<\sup_{x\in D}\varphi_{2}\left(x\right)$.
The reason being that if the Hot-Spots Conjecture holds for $D$,
then it clear that (\ref{eq:Thm-HS-1}) holds with $C=1.$ Thus there
exists a $x_{0}\in D$ such that $\sup_{x\in D}\varphi_{2}\left(x\right)=\varphi_{2}\left(x_{0}\right).$
By repeating the same arguments as in \cite[Lemma 1]{Steinerberger-2021a}  or \cite[Lemma 4]{Mariano-Panzo-Wang-2021} for
reflecting Brownian motion on a Riemannian manifold, one can obtain
\begin{equation}\label{eq:HotSpotpf-1}
1\leq e^{\mu_{1}\left(D\right)t}\mathbb{P}_{x_{0}}\left(\tau_{D}>t\right)+e^{\mu_{1}\left(D\right)t}\left(1-\mathbb{P}_{x_{0}}\left(\tau_{D}>t\right)\right)\frac{\sup_{x\in\partial D}\varphi_{2}\left(x\right)}{\sup_{x\in D}\varphi_{2}\left(x\right)}.
\end{equation}
By Theorem \ref{thm:Prob-Estimate} and the verifications given in Section \ref{sec:subRiemannian} we have that for any $\epsilon\in\left(0,1\right)$ (to be chosen
later) there exists a constant $C_{\epsilon}>0$ such that
\begin{equation}\label{eq:HotSpotpf-2}
\sup_{x\in D}\mathbb{P}_{x}\left(\tau_{D}>t\right)\leq C_{\epsilon}e^{-\left(1-\epsilon\right)t\lambda_{1}\left(D\right)}
\end{equation}
for any $t>0$. Plugging \eqref{eq:HotSpotpf-2} into \eqref{eq:HotSpotpf-1} and applying the assumption that $\frac{\mu_{1}\left(D\right)}{\lambda_{1}\left(D\right)}\leq \hs<1$
for all $D\in\mathcal{D}_\hs$ we have that
\[
1\leq e^{\hs\lambda_{1}\left(D\right)t}C_{\epsilon}e^{-\left(1-\epsilon\right)t\lambda_{1}\left(D\right)}+e^{\hs\lambda_{1}\left(D\right)t}\left(1-C_{\epsilon}e^{-\left(1-\epsilon\right)t\lambda_{1}\left(D\right)}\right)\frac{\sup_{x\in\partial D}\varphi_{2}\left(x\right)}{\sup_{x\in D}\varphi_{2}\left(x\right)}.
\]
Pick $\epsilon$ small enough such that $\hs<\left(1-\epsilon\right)<1$
and let $t=\frac{a}{\lambda_{1}(D)}$ for some $a>0$ to be chosen
later. Then 
\begin{equation}
1\leq C_{\epsilon}e^{-\left(\left(1-\epsilon\right)-\hs\right)a}+e^{\hs a}\left(1-C_{\epsilon}e^{-\left(1-\epsilon\right)a}\right)\frac{\sup_{x\in\partial D}\varphi_{2}\left(x\right)}{\sup_{x\in D}\varphi_{2}\left(x\right)}.\label{eq:pf:HotSpot-1}
\end{equation}
Now pick $a>0$ big enough so that 
\[
1-C_{\epsilon}e^{-\left(\left(1-\epsilon\right)-\hs\right)a}>0\text{ and }1-C_{\epsilon}e^{-\left(1-\epsilon\right)a}>0
\]
hence by rearranging (\ref{eq:pf:HotSpot-1}) we can then obtain the
following bound
\[
\sup_{x\in D}\varphi_{2}\left(x\right)\leq e^{\hs a}\frac{1-C_{\epsilon}e^{-\left(1-\epsilon\right)a}}{1-C_{\epsilon}e^{-\left(\left(1-\epsilon\right)-\hs\right)a}}\sup_{x\in\partial D}\varphi_{2}\left(x\right),
\]
as desired. 
\end{proof}

The assumption \eqref{eq:Dirch-Neuman-constant} has been classically known to hold for all domains in the Euclidean setting. Next we give an example of this assumption holding on $S^d$.

\begin{example}[Domains in $S^d$]
Let $H\subset S^{d}$ be a Hemisphere in $S^{d}$.
By \cite[Equation (2.16)]{Ashbaugh-Levine-1997} it was shown that 
\begin{equation}
\frac{\mu_{2}\left(D\right)}{\lambda_{1}\left(D\right)}\leq\frac{\mu_{2}\left(D^{\star}\right)}{\lambda_{1}\left(D^{\star}\right)}<1\label{eq:Hot-Spot-sphere-1}
\end{equation}
for all domains $D\subset H\subset S^{d}$ such that $\left|D\right|=\left|D^{\star}\right|$
where $D^{\star}$ is a geodesic ball in $S^{d}$ contained in $H$.
As explained in \cite{Ashbaugh-Levine-1997}, the inequality $\frac{\mu_{2}\left(D\right)}{\lambda_{1}\left(D\right)}<1$
cannot hold for all smooth domains in $S^{d}$. This is because $\frac{\mu_{2}\left(B\right)}{\lambda_{1}\left(B\right)}=1$
when the geodesic ball $B$ is a Hemisphere, and since $\frac{\mu_{2}\left(B\right)}{\lambda_{1}\left(B\right)}\to\infty$
as a ball $B$ approaches the full sphere $S^{d}$. Fix $\hs\in\left(0,1\right)$
and let $\mathcal{D}_{\hs}$ be the class of all domains $D\subset H$
such that $\frac{\mu_{2}\left(D\right)}{\lambda_{1}\left(D\right)}<\hs$. By \eqref{eq:Hot-Spot-sphere-1} we know that $\mathcal{D}_{\hs}$ is non-empty. 
We can then conclude that Proposition \ref{thm:Hot-Spots}
holds for the class $\mathcal{D}_{\hs}$ of domains in $S^{d}$. 
\end{example}


\section{A new characterization of heat kernel upper bounds\label{sec:3-MainResults-b}}

As an application  of Theorem \ref{thm:Prob-Estimate} and Proposition \ref{Cor:main_product_bound}, in this section we prove equivalent conditions for the heat kernel upper estimate $\left(\UEF\right)$. 
There has been great interest by numerous authors in proving heat kernel bounds, equivalent conditions to heat kernel bounds and other related results in a metric measure space setting. See \cite{Besov2-2020,Grigoryan-Hu-Lau-2014,Kigami-2004,Lou-2018,Telcs-1985,Chen-Kim-Song-2010}, to name a few. 

We start by stating some conditions on $\lambda(D)$. Since domain monotonicity holds for $\lambda(D)$  it is natural in many cases to consider growth estimates for the first Dirichlet eigenvalues of balls. 
\begin{df}[$\lambF,\lambFge,\lambFle$]\label{def:lambda-beta} The metric measure space $(M, d,\mu)$ endowed with a Dirichlet form $(\cE, \cF)$ is said to satisfy  $\left(\lambF\right)$ if there exists constants $c_1,c_2>0$ such that for almost every $x\in M$ and $r>0$
\begin{equation}
c_1 F\left(r\right)^{-1} \leq \lambda\left(B\left(x,r\right)\right)\leq c_2 F\left(r\right)^{-1}.\label{Lambda-beta}
\end{equation}
It is said to satisfy the condition $\left(\lambFle\right)$ (and $\left(\lambFge\right)$ respectively) if
\[
 \lambda\left(B\left(x,r\right)\right)\leq c_2 F\left(r\right)^{-1}\  \text{and } \lambda\left(B\left(x,r\right)\right)\geq c_1 F\left(r\right)^{-1}\ \text{holds respectively.}
\]

If $F\left(r\right)=r^\beta$ then we call these conditions $\left(\lamb\right)$, $\left(\lamble\right)$, $\left(\lambge\right)$. 
\end{df}

The following generalizes the notion of a Faber-Krahn inequality for metric measure spaces.

\begin{df}[$\FKF$]\label{df:FK} The metric measure space $(M, d,\mu)$ endowed with a Dirichlet form $(\cE, \cF)$  is said to satisfy the \emph{Faber-Krahn inequality} with parameter function $F$ if there exist constants $\nu,c>0$ such that, for all balls $B\subset M$
of radius $r>0$ and for any nonempty open sets $D\subset B$, 
\begin{equation}
\lambda(D)\geq\frac{c}{F\left(r\right)}\left(\frac{\mu\left(B\right)}{\mu\left(D\right)}\right)^{\nu}.\label{FK}
\end{equation}
If $F\left(r\right)=r^\beta$ then we call this condition $\left(\FK\right)$. 
\end{df}

Below we state several conditions that concerns the growth of mean exit times from balls.
\begin{df}[$\EF$, $\EFge$, $\EFle$]\label{def-E-beta}
The metric measure space $(M, d,\mu)$ endowed with a Dirichlet form $(\cE, \cF)$ satisfies the condition $\left(\EF\right)$ with parameter function $F$ if there exists $c_1,c_2>0$ and a properly exceptional set $\mathcal{N}$ such that for all $x\in M\backslash\mathcal{N}$
and $r>0$
\begin{equation}
c_1 F\left(r\right) \leq \mathbb{E}_{x}\left[\tau_{B(x,r)}\right]\leq c_2 F\left(r\right).\label{E-beta}
\end{equation}
It is said to satisfy the condition $\left(\EFle\right)$ (and $\left(\EFge\right)$ respectively) if
\[
\mathbb{E}_{x}\left[\tau_{B(x,r)}\right]\leq c_2 F\left(r\right)\  \text{and } \mathbb{E}_{x}\left[\tau_{B(x,r)}\right]\geq c_1 F\left(r\right)\ \text{holds respectively.}
\]

If $F\left(r\right)=r^\beta$ then we call these conditions $\left(\Eb\right)$, $\left(\Eble\right)$, $\left(\Ebge\right)$.

\end{df}
The condition stated below is an isoperimetric-type inequality that concerns the growth of the mean exit time of $X_t$ from a domain.

\begin{df}[$\EOF$]\label{def-EO}
 We say that the metric measure space $(M, d,\mu)$ endowed with a Dirichlet form $(\cE, \cF)$ satisfy the condition $\left(\EOF\right)$ with parameter function $F$ if there exists a properly exceptional set $\mathcal{N}\subset M$ and
positive constants $C,\nu$ such that, for all balls $B$ in $M$
of radius $r$ and any non-empty open sets $D\subset B$, 
\[
\sup_{x\in D\setminus\mathcal{N}}\mathbb{E}_{x}\left[\tau_{D}\right]\leq C F(r)\left(\frac{\mu\left(D\right)}{\mu\left(B\right)}\right)^{\nu}.
\]
If $F(r)=r^\beta$ then we call this condition $\left(\EO\right)$.

\end{df}

Lastly we introduce a condition that restricts the optimal starting point for a Markov process for the purpose of prolonging its lifetime in a ball. It shall be used in obtaining equivalent conditions for  $\left(\UE\right)$  in Proposition \ref{MainThm2}.

\begin{df}[$\Ebar$]\label{def:e-bar} The metric measure space $(M, d,\mu)$ endowed with a Dirichlet form $(\cE, \cF)$ satisfies the condition $\left(\Ebar\right)$ if there exists a properly exceptional set $\mathcal{N}\subset M$ and
a positive constant $C>0$ such that for all $x\in M\backslash\mathcal{N}$
and $r>0$, 
\begin{equation}\label{eq:e-bar}
C\sup_{y\in B\left(x,r\right)\backslash\mathcal{N}}\mathbb{E}_{y}\left[\tau_{B\left(x,r\right)}\right]\leq\mathbb{E}_{x}\left[\tau_{B\left(x,r\right)}\right].
\end{equation}
\end{df}

This condition essentially says that the center of a ball is the optimal starting point for prolonging the mean exit time. We also introduce the following ``Harnack-type inequality" for mean exit times. 
\begin{df}[$\Ebarp$]\label{df:e-bar-prime}
The metric measure space $\left(M,d,\mu\right)$ endowed with a Dirichlet
form $\left(\mathcal{E},\mathcal{F}\right)$ satisfies the condition
$\left(\Ebarp\right)$ if there exists a constant
$C>0$ and $\delta\in\left(0,1\right)$ such that for all
$r>0$ and all $x\in M$
\begin{equation}
C\esup_{y\in B\left(x, r\right)}\mathbb{E}_{y}\left[\tau_{B\left(x,r\right)}\right]\leq\einf _{y\in B(x,\delta  r)}\mathbb{E}_{y}\left[\tau_{B\left(x,r\right)}\right].
\end{equation}
\end{df}

We also have the following condition on the metric measure space. 
\begin{df}[$\RVD$]\label{def-RVD}  The metric measure space $(M,d,\mu)$ is said to satisfy the
\emph{reverse volume doubling property with parameter $\alpha'$} if there exist a constant $c_{\alpha^\prime}>0$ such that 
\begin{equation}
\frac{V(x,R)}{V(x,r)}\geq  c_{\alpha^\prime}
\left(\frac{R}{r}\right)^{\alpha^{\prime}}\text{ for all }x\in M\text{ and for all }0<r\leq R.\label{eq:RVD-Estimate}
\end{equation}
\end{df}

As another interesting application  of Theorem \ref{thm:Prob-Estimate}, the result below provides a set of equivalent conditions for the sub-Gaussian upper heat kernel bound that involves the Faber-Krahn inequality, the regular spectral growth on balls, and a Harnack-type inequality for exit times. 

\begin{prop}\label{MainThm2}
Let $\left(M,d,\mu\right)$ be a metric measure space
and assume $\mu$ satisfies $\left(\VD\right)$ and $\left(\RVD\right)$ for some $\alpha, \alpha'>0$. Let $\left(\mathcal{E},\mathcal{F}\right)$
be a regular, local, conservative Dirichlet form in $L^{2}\left(M,\mu\right)$.
Then, the following equivalence is true:
\begin{align}
\left(\UEF\right) & \iff\left(\FKF\right)+\left(\lambFle\right)+\left(\Ebar\right)\label{eq:Main5}\\
 & \iff\left(\FKF\right)+\left(\lambFle\right)+\left(\Ebarp\right).\label{eq:Main5b}
\end{align}
\end{prop}
\begin{remark}
Note that $\left(\lambFle\right)$ is not necessary in order to obtain the equivalence to $\left(\UEF\right)$. Though we point out that \eqref{eq:Main5} and \eqref{eq:Main5b} are true when $\left(\lambFle\right)$ is replaced with $\left(\lambF\right)$. This is explained in Lemma \ref{prop:Main-Prop-1}.
\end{remark}
\begin{remark}
In \cite{Grigoryan-Hu-2014}, several sets of  equivalent conditions to sub-Gaussian heat kernel bounds $\left(\UE\right)$ are given. In particular they obtain that 
\begin{align}
\left(\UE\right) 
 & \iff
 \left(\EO\right)+\left(\Eb\right).\label{thm:Grigoryan-Hu-3}
\end{align}
 In \cite{Grigoryan-Hu-Lau-2014} the authors asked if the spectral analogue of the above equivalence,
  \[
  \left(\UE\right)  \iff\left(\FK\right)+\left(\lamb\right)
  \]
  could hold. This question is still open. 
 Roughly speaking, $\left(\Ebar\right), \left(\Ebarp\right)$  in \eqref{eq:Main5} and \eqref{eq:Main5b} says that the optimal location for the maximal mean exit time from a ball is close to the center of the ball.  It is not clear whether $\left(\Ebar\right), \left(\Ebarp\right)$ are removable in a different approach. Though we point out that it is a natural condition and have already been considered in the works of \cite{Telcs-2006-Ebar2,Telcs-2006-Ebar1} in the setting of graphs. 
\end{remark}


We devote the rest of this section to the proof of Proposition \ref{MainThm2}, which is split into the Lemmas \ref{prop:Main-Prop-1}, \ref{lem:UE_implies_E_bar_prime} and \ref{lem:E-beta-lower}.

\begin{lemma}
\label{prop:Main-Prop-1}
Under the same assumptions of Proposition \ref{MainThm2}, we have that 
\[
\left(\UEF\right)
\implies\left(\FKF\right)+\left(\lambF\right).
\]
\end{lemma}
\begin{proof}
Assume $\left(\UEF\right)$ holds.  It follows from \cite[Theorem 2.1]{Grigoryan-Hu-2014}  that under the same assumptions of Proposition \ref{MainThm2}
\[
\left(\UEF\right) \iff\left(\FKF\right)+\left(\text{S}_{F}\right)\label{thm:Grigoryan-Hu-2},
\]
where $\left(\text{S}_{F}\right)$ means  there exists an $\eta,\eta^\prime\in\left(0,1\right)$ 
such that for all $t>0$ and all balls $B=B\left(x,r\right)$ with
$t\leq\eta^{\prime}F\left(r\right)$, one has 
\begin{equation}
\eta\leq P_{t}^{B}1_{B}\left(y\right),\text{ for }\mu-\text{a.e. }y\in\frac{1}{4}B.\label{eq:lambda_beta-upper-1}
\end{equation}
We now only need to show that $\left(\UEF\right)$  implies $\left(\lambF\right)$ to finish the proof of the Lemma.

We first show the upper bound of $\left(\lambF\right)$. By Theorem \ref{thm:Prob-Estimate} with $\epsilon=1/2$, we know there there exists $C>0$ such that
for any domain $D\subset M$ we have 
\begin{equation}
P_{t}^{D}1_{D}\left(y\right)\leq Ce^{-\lambda\left(D\right)t/2},\label{eq:lambda_beta-upper-2}
\end{equation}
for almost every $y\in D$. Applying (\ref{eq:lambda_beta-upper-2})
to $B\left(x,r\right)$ and using (\ref{eq:lambda_beta-upper-1})
we have that for all balls, 
\[
\eta\leq Ce^{-\lambda\left(B\left(x,r\right)\right)t/2},
\]
as long as $t\leq\eta^{\prime}F\left(r\right)$. Integrating both sides
from $0$ to $\eta^{\prime}F\left(r\right)$ we have 
\[
\eta\cdot \eta^{\prime}F\left(r\right)  \leq\frac{2C}{\lambda\left(B\left(x,r\right)\right)}\left(1-e^{-\lambda\left(B\left(x,r\right)\right)\eta^{\prime}F\left(r\right)/2}\right),
\]
which implies that 
\begin{equation}\label{eq:Main-Prop1-eq2}
\lambda\left(B\left(x,r\right)\right)\leq2C\eta^{-1}(\eta^{\prime})^{-1}F\left(r\right)^{-1}.
\end{equation}


The lower bound estimate on $\lambda\left(B\left(x,r\right)\right)$ follows immediately by $\left(\FKF\right)$. Indeed, by taking $D=B$ in $\left(\FKF\right)$ one  has that there exists constants $c,\nu>0$ such that for all balls
$B\subset M$ of radius $r>0$ , we have 
\begin{equation}\label{eq:Main-Prop1-eq3-new}
\lambda\left(B\right)\geq\frac{c}{F\left(r\right)}\left(\frac{\mu\left(B\right)}{\mu\left(B\right)}\right)^{\nu}=\frac{c}{F\left(r\right)},
\end{equation}

Combining \eqref{eq:Main-Prop1-eq2} and \eqref{eq:Main-Prop1-eq3-new} we then obtain $\left(\lambF\right)$.
\end{proof}
%


The idea of the following is similar to the one found in \cite[Lemma 4.3]{Telcs-2001b}
\begin{lemma}
\label{lem:UE_implies_E_bar_prime}Under the same assumptions of Proposition \ref{MainThm2} we have that 
\begin{align*}
\left(\EF\right) \implies\left(\Ebar\right)\quad\text{and}\quad 
\left(\EF\right)  \implies\left(\Ebarp\right).
\end{align*}
\end{lemma}

\begin{proof}


%


From $\left(\EF\right)$ we know that there exist a
properly exceptional set $\mathcal{N}$ and constants $c_{1},c_{2}>0$
such that for all $x\in M\backslash\mathcal{N}$ and $r>0$ we have
\begin{equation}\label{eq:E-bar-prime-1}
c_{1} F\left(r\right)\leq\mathbb{E}_{x}\left[\tau_{B(x,r)}\right]\leq c_{2}F\left(r\right).
\end{equation}
By a coupling argument we have that for any given $x\in M$ and $r>0$ we have
\begin{align}
\mathbb{E}_{y}\left[\tau_{B\left(x,r\right)}\right] & \leq\mathbb{E}_{y}\left[\tau_{B\left(y,2r\right)}\right].\label{eq:coupling}
\end{align}
for any $y\in B\left(x,r\right)$.
Hence
\begin{equation}\label{eq:Main-Prop-1-eq1e}
\sup_{y\in B\left(x,r\right)\setminus\mathcal{N}}\mathbb{E}_{y}\left[\tau_{B\left(x,r\right)}\right]\leq\sup_{y\in B\left(x,r\right)\setminus\mathcal{N}}\mathbb{E}_{y}\left[\tau_{B\left(y,2r\right)}\right]\leq c_{2}F\left(2r\right) \leq c_{2}C_F 2^{\beta^\prime}F\left(r\right).
\end{equation}
By \eqref{eq:Main-Prop-1-eq1e} combined with \eqref{eq:E-bar-prime-1} we have that for $x\in M\setminus\mathcal{N}$,
\begin{align*}
\sup_{y\in B\left(x,r\right)\setminus\mathcal{N}}\mathbb{E}_{y}\left[\tau_{B\left(x,r\right)}\right] & \leq c_{2}C_F 2^{\beta^\prime}F\left(r\right)
  \leq 2^{\beta^\prime}C_F c_{2}c_{1}^{-1}\mathbb{E}_{x}\left[\tau_{B\left(x,r\right)}\right],
\end{align*}
so that $\left(\Ebar\right)$ holds.


To see the second implication, by combining \eqref{eq:Main-Prop-1-eq1e} and \eqref{eq:E-bar-prime-1} we have that for $\mu-$almost every $z\in B\left(x,\frac{r}{2}\right)$,
\begin{align*}
 & \esup_{y\in B\left(x,r\right)}\mathbb{E}_{y}\left[\tau_{B\left(x,r\right)}\right]  \leq c_{2}C_F 2^{\beta^\prime}F\left(r\right)\\
 &\quad\quad \leq4^{\beta^{\prime}}C_{F}^{2}c_{2}c_{1}^{-1}\mathbb{E}_{z}\left[\tau_{B\left(z,\frac{r}{2}\right)}\right]
  \leq 4^{\beta^{\prime}}C_{F}^{2}c_{2}c_{1}^{-1}\mathbb{E}_{z}\left[\tau_{B\left(x,r\right)}\right],
\end{align*}
where the last inequality follows from $B\left(z,\frac{r}{2}\right)\subset B\left(x,r\right)$ for all $z\in B\left(x,\frac{r}{2}\right)$. 
Hence we have 
\[
\esup_{y\in B\left(x,r\right)}\mathbb{E}_{y}\left[\tau_{B\left(x,r\right)}\right] \leq 4^{\beta^{\prime}}C_{F}^{2}c_{2}c_{1}^{-1}\einf_{z\in B\left(x,\frac{r}{2}\right)}\mathbb{E}_{z}\left[\tau_{B\left(x,r\right)}\right]
\]
so that $\left(\Ebarp\right)$ holds with $\delta=\frac{1}{2}$.
\end{proof}

The results below handles the backward implications of Proposition \ref{MainThm2}.

\begin{lemma}
\label{lem:E-beta-lower}Under the same assumptions of Proposition \ref{MainThm2},
we have that
\[
\left(\FKF\right)+\left(\lambFle\right)+\left(\Ebar\right)\implies\left(\EFge\right).
\]
\end{lemma}

\begin{proof} 
By \eqref{eq:Lower-Bound} along with $\left(\lambFle\right)+\left(\Ebar\right)$
we have that there exists a properly exceptional set $\mathcal{N}$
such that for all $r>0$ and all $x\in M\setminus \mathcal{N}$,
\begin{align*}
1  \leq\lambda\left(B\left(x,r\right)\right)\esup_{y\in B\left(x,r\right)}\mathbb{E}_{y}\left[\tau_{B\left(x,r\right)}\right]
\le c_2F(r)^{-1} \sup_{y\in B\left(x,r\right)\setminus \mathcal{N}}\mathbb{E}_{y}\left[\tau_{B\left(x,r\right)}\right]
 \leq \frac{c_2}{C}F(r)^{-1}\mathbb{E}_{x}\left[\tau_{B\left(x,r\right)}\right].
\end{align*}
Therefore we have the conclusion.

\end{proof}


We are now ready to prove Proposition \ref{MainThm2}.

\begin{proof}[\textbf{Proof of Proposition \ref{MainThm2}}]
We first show \eqref{eq:Main5}. By Lemmas \ref{prop:Main-Prop-1} and \ref{lem:UE_implies_E_bar_prime} and the fact that 
\begin{equation}\label{eq-UE-FK-EF}
(\UEF)\implies (\FKF)+(\EF),
\end{equation}
(see \cite{Grigoryan-Hu-Lau-2015, Grigoryan-Hu-2014})  we have that $(\UEF)\implies \left(\FKF\right)+\left(\lambFle\right)+\left(\Ebar\right)$. To see the backward implication, 
consider the condition $\left(\mathbf{E}_{F\geq}^\prime\right)$: there exists a constant
$C>0$ and $\delta\in\left(0,1\right)$ such that for all
$r>0$ and all $x\in M$
\begin{equation}
CF\left(r\right)\leq\einf _{y\in B(x,\delta  r)}\mathbb{E}_{y}\left[\tau_{B\left(x,r\right)}\right].
\end{equation}
It is known from \cite[Theorem 2.2]{Grigoryan-Hu-2014} with \cite[Equation (9.11)]{Grigoryan-Hu-Lau-2015} that $\left(\FKF\right)+\left(\mathbf{E}_{F\geq}^\prime\right) \implies\left(\UEF\right)$. 
Using an argument similar to that of the second part of the proof of Lemma \ref{lem:UE_implies_E_bar_prime} one can easily show that 
\begin{equation}
\left(\EFge\right)\implies\left(\mathbf{E}_{F\geq}^\prime\right).\label{eq:thm3.11-0}
\end{equation}
Thus we have 
\begin{align}
\left(\FKF\right)+\left(\EFge\right) & \implies\left(\FKF\right)+\left(\mathbf{E}_{F\geq}^\prime\right) \implies\left(\UEF\right)\label{eq:thm3.11-1}
\end{align}
Combining with Lemma \ref{lem:E-beta-lower} we then obtain the desired backward implication. 


To prove \eqref{eq:Main5b}, again by combining Lemmas \ref{prop:Main-Prop-1}, \ref{lem:UE_implies_E_bar_prime} and \eqref{eq-UE-FK-EF} we obtain the forward implication. 
Moreover, by using a similar argument as in the proof of Lemma \ref{lem:E-beta-lower} we obtain that 
\begin{equation}\label{eq:thm3.11-2}
\left(\FKF\right)+\left(\lambFle\right)+\left(\Ebarp\right)\implies\left(\mathbf{E}_{F\geq}^\prime\right).
\end{equation}
The backward implication then follows easily from \eqref{eq:thm3.11-2} and \eqref{eq:thm3.11-1}.
\end{proof}


\section{Examples\label{sec:7-Applications}}

The results in previous sections apply to a broad class of metric measure spaces, including Riemannian and sub-Riemannian manifolds as well as fractals. We demonstrate some examples for each of these settings.



\subsection{The Heisenberg group}

\label{ex-Heis}
The Heisenberg group $\Heis$ is the simplest non-trivial example of a sub-Riemannian manifold. It can be identified with $\R^{2n}\times \R$, equipped with the group law: 
\[
\left(\mathbf{x},z\right)\star\left(\mathbf{x}',z'\right):=\left(\mathbf{x}+\mathbf{x}',z+z'+\sum_{i=1}^{n}x_{i}y'_{i}-x'_{i}y_{i}\right).
\]
where $\mathbf{x}=\left(x,y\right)\in(\mathbb{R}^{n})^{\otimes 2},\mathbf{x}'=\left(x',y'\right)\in(\R^n)^{\otimes 2}$.
The identity in $\Heis$ is $e=(\mathbf{0},0)$ and the inverse is given by $(\mathbf{x},z)^{-1}=(-\mathbf{x},-z)$. Let $\mathfrak{h}$ denote its Lie algebra. It can be identified  with the tangent space $T_e(\Heis)$. A basis of left invariant vector fields at $p=(\mathbf{x},z)$ can be given by
\begin{equation}\label{eq-basis}
X_i(p)=\partial_{x_i}+y_i\partial_{z}, \quad Y(p)=\partial_{y_i}-x_i\partial_{z}, \quad Z(p)=\partial_{z}, \quad i=1,\dots, n.
\end{equation}
Let $\mathcal{H}_p:=\mathrm{Span}\{X_i(p), Y_i(p), i=1,\dots, n\}$, $p\in \Heis$ be a horizontal bundle. Clearly it satisfies the H\"ormander's condition since $[X_i,Y_i]=Z$ for all $1\le i\le n$.
We equip  the horizontal bundle $\ch_p$, $p\in \Heis$ with an inner product such that $\{X_i(p), Y_i(p), i=1,\dots, n\}$ is an orthonormal frame at $p$. A path $\gamma:[0,1]\to \Heis$ is called horizontal if it is absolutely continuous and $\dot{\gamma}(t)\in \ch_{\gamma(t)}$ for all $t\in [0,1]$. The Carnot-Carath\'eodory distance on $\Heis$ is then defined as
\begin{equation}\label{eq-cc-dist}
d_{cc}(p,p'):=\inf\bigg\{\int_0^1 |\dot{\gamma}(t)|_{\ch}dt \,
 ; \, 
 \gamma:[0,1]\to \Heis \ \mbox{is horizontal}, \ \gamma(0)=p,\gamma(1)=p'\bigg\}.
\end{equation}
We equip $\Heis$ with the Haar measure $\mu$. Let $B(p,r)$ denote the  ball centered at $p\in\Heis$ with radius $r>0$ (with respect to $d_{cc}$). Then one can easily observe that for any $c>0$,
\[
\mu\left(B\left(p,cr\right)\right)=c^{2n+2}\mu\left(B\left(p,r\right)\right).
\]
Clearly both $\left(\VD\right)$ and $\left(\RVD\right)$ are satisfied with $\alpha=\alpha'=2n+2$.

Consider the sub-Laplacian associated to the sub-Riemannian structure on $\Heis$:
\begin{equation}\label{eq-Laplacian}
\Delta_\ch=\sum_{i=1}^nX_i^2+Y_i^2.
\end{equation}
It is an essentially self-adjoint operator with respect to $\mu$. It is hypoelliptic and admits a smooth heat kernel 
$p_{\Heis}$ with respect to $\mu$. From \cite{HQLi-2006} and other authors we know that $\left(\UE\right)$ is satisfied with $\beta=2$.

We call the Markov process $\{\mathcal{B}_t\}_{t\ge0}$ generated by $\frac12\Delta_\ch$ the horizontal Brownian motion on $\Heis$. Let $(U, V)$ be a standard Brownian motion  in $\R^{2n}$, then $\{\mathcal{B}_t\}_{t\ge0}$ started from $\mathcal{B}_0=e$ have the following presentation 
\begin{equation}\label{eq-BM}
\BH_t=\left(U_t, V_t,A_t\right) ,
\quad\text{where}\quad
A_t=\sum_{i=1}^n\int_0^t{U^i_s}dV^i_s-V^i_sdU_s^{i} .
\end{equation}
The density for the distribution of $\BH_t$ at any $t>0$ is given by the heat kernel $p_{\Heis}$. 

For any given domain $D\subset\Heis$, by \cite[Proposition 4.1]{Chen-Chen-LMS-2021}  with smooth and non-characteristic boundary, there exists a smooth Dirichlet heat kernel $p_{D}$ such that
$p_{D}\left(p,p',t\right)>0$ for all $\left(p,p',t\right)\in\Heis\times\Heis\times\left(0,+\infty\right)$. We conclude that the results of Theorem \ref{thm:Prob-Estimate}, Corollary \ref{Cor:polynomial-bound}, Proposition \ref{Cor:main_product_bound}, the upper bound  \eqref{eq:Asymp1'} of Proposition \ref{thm:Main_Asymptotic}, and Proposition \ref{MainThm2} all hold.


\subsection{Carnot Groups}\label{sec:Carnot}

Let $M$ be a $N$-step Carnot group. It is a simply connected
Lie group with Lie algebra given by
\[
\mathfrak{g}=\mathcal{V}_{1}\oplus\cdots\oplus\mathcal{V}_{N}
\]
where $\mathcal{V}_{i}$, $i\ge1$ are subspaces of $\mathfrak{g}$ satisfying $\left[\mathcal{\mathcal{V}}_{i},\mathcal{V}_{j}\right]=\mathcal{V}_{i+j}$
for all $i,j\geq1$ and $\mathcal{V}_{s}=0$ for $s>N$. The Heisenberg group $\Heis$ (see Section \ref{ex-Heis}) is a special case of a $2$-step Carnot group.

Assume $\dim\mathcal{V}_{1}=d.$ Let $\{V_1,\dots, V_d\}$ be a basis that spans $\mathcal{V}_{1}$. Denote by $V_i(p)$, $p\in M$, $i=1,\dots, d$ the corresponding left invariant vector fields. Let $\ch_p:=\mathrm{Span}\{V_i(p),i=1,\dots, d\}$. Then $\mathcal{H}_p$, $p\in M$ is a horizontal bundle on $M$. We consider a sub-Riemannian structure on $(M, \ch)$ by setting an inner product such that $\{V_i(p),  i=1,\dots, d\}$ is an orthonormal frame at $p$. It then induces a C-C distance $d_{cc}$ similarly as in \eqref{eq-cc-dist}. Let $\mu$ be the Haar measure on $M$. It is known that the metric measure space $(M, d_{cc}, \mu)$ satisfies $\left(\VD\right)$ and $\left(\RVD\right)$. In fact, by dilation arguments we can easily see that 
\[
\mu\left(B\left(x,R\right)\right)=CR^{Q}
\]
where $B\left(x,R\right)$ is the C-C ball centered at $x\in M$ of radius $R>0$, $C=\mu\left(B\left(x,1\right)\right)$ and $Q=\sum_{j=1}^{N}i\text{dim}\mathcal{V}_{i}$ is the homogeneous dimension of $M$. Therefore $\alpha=\alpha'=Q$.

There is a natural choice of Dirichlet form associated to the sub-Riemannian structure on $M$ whose generator is given by the sub-Laplacian (see for example \cite[Section 3, pp. 233-234]{Sturm-1995})
\[
L=\sum_{i=1}^{d}V_{i}^{2}.
\]
We denote by $p\left(x,y,t\right)$, $(x,y,t)\in M\times M\times \R_+$ the heat kernel of $L$ with respect to $\mu$. The heat kernel is smooth by H\"ormander's theorem and satisfies
\begin{equation}
p_M\left(x,y,t\right)\leq\frac{C}{t^{Q/2}}\exp\left(-c\frac{d_{cc}\left(x,y\right)^{2}}{t}\right)\label{eq:Carnot-pBound}
\end{equation}
for some constants $C,c>0$. Hence $\left(\UE\right)$ is satisfied with $\beta=2$. In fact, it is well
known result of double-sided Gaussian bounds for the heat kernel (see reference \cite[
Theorem VIII2.9]{Varopoulos-etall-1992}). 

Therefore Theorem \ref{thm:Prob-Estimate},  Corollary \ref{Cor:polynomial-bound}, Proposition 
 \ref{Cor:main_product_bound}, and Proposition \ref{MainThm2} all hold. The upper bound  \eqref{eq:Asymp1'} of Proposition  \ref{thm:Main_Asymptotic} also holds.

If in addition, we are given a bounded domain $D\subset M$ with smooth and non-characteristic boundary then by \cite[Sec. 4, Proposition 4.1]{Chen-Chen-LMS-2021} the Dirichlet heat kernel $p_D$ exists, is strictly positive and smooth. By Proposition \ref{prop:Irreducibility_cond} the Dirichlet form $\left(\mathcal{E},\mathcal{F}(D)\right)$ is irreducible. Hence \eqref{eq:Asymp1} of Proposition \ref{thm:Main_Asymptotic} holds.

\begin{remark}
In results needing the irreducibility condition, we believe the non-characteristic assumption on domains is not necessary. We point to the example of the work of Carfagnini-Gordina in \cite{Carfagnini-Gordina-2020, Carfagnini-Gordina_Carnot-2022} where they are able to handle general bounded open sets on Carnot groups. In their work, they are able to show the exit time asymptotic \eqref{eq:Asymp1} for bounded open sets for Carnot groups. Moreover, they show that $\mathcal{L}_D$ has a discrete spectrum for any bounded open set, and hence the limit in \eqref{eq:Asymp1} is shown to be the first Dirichlet eigenvalue. 
\end{remark}

\subsection{Riemannian and sub-Riemannian manifolds}\label{sec:subRiemannian}

Let $(M,g)$ be a connected complete Riemannian manifold and $\mu$ be the corresponding Riemannian volume measure. Let  $\Delta$ be the Laplace-Beltrami operator on $M$, and $p(x,y,t)$ be the heat kernel associated to the heat semigroup generated by $\Delta$.

If the Ricci curvature of $M$ is bounded from below by a non-negative constant,  it is known that $(M,g, \mu)$ admits an upper and lower heat kernel Gaussian bounds. Namely for all $x,y\in M$ and $t>0$,
\[
\frac{1}{C\mu(B(x,\sqrt{t}))}\exp\left(-\frac{d(x,y)^2}{(4-\epsilon)t} \right) \le p_M(x,y,t)\le \frac{C}{\mu(B(x,\sqrt{t}))}\exp\left(-\frac{d(x,y)^2}{(4+\epsilon)t} \right)
\]
for any $0<\epsilon<1$, where $C>0$ depends on $\epsilon$ and on $M$.

Another large class of examples is the \textit{sub-Riemannian manifolds with transverse
symmetries}. Let $M$ be an $n$-dimensional smooth complete manifold equipped with a bracket generating distribution $\mathcal{H}$ of dimension $m<n$. Let $g$ be a fiber-wise inner product on $\mathcal{H}$. Let $\mathcal{V}$ denote the finite-dimensional Lie algebra of all sub-Riemannian Killing vector fields on $M$. Transverse symmetry means that for every $Z\in \mathcal{V}$ the Lie derivative of $g$ with respect to $Z$ is constantly $0$, and $[Z,X]\in \mathcal{H}$ for all $X\in \mathcal{H}$.

Let $L$ be the sub-Laplacian on $M$ and $p(x,y,t)$ the associated heat kernel. Assume $L$ satisfies a generalized curvature-dimension inequality $CD(\rho_1, \rho_2, \kappa, m)$ that is introduced by Baudoin-Garofalo \cite{Baudoin-Garofalo-2017}, where $\rho_1\in \R$, $\rho_2>0$, $\kappa\ge 0$ are geometric bounds for certain sub-Riemannian tensors. Then we know that  (see \cite[ Theorem 4.1]{BaudoinBonnefontGarofalo2014})  given $\rho_1\ge0$,  $\left(\ULE\right)$ is satisfied in the form of
\[
\frac{1}{C\mu(B(x,\sqrt{t}))}\exp\left(-\frac{Dd(x,y)^2}{m(4-\epsilon)t} \right) \le p_M(x,y,t)\le \frac{C}{\mu(B(x,\sqrt{t}))}\exp\left(-\frac{d(x,y)^2}{(4+\epsilon)t} \right).
\]
where $C=C(m, \kappa, \rho_2,\epsilon)>0$ and $D=m\left(1+\frac{3\kappa}{2\rho_2}\right)$. It is also shown that the volume doubling condition $\left(\VD\right)$ is satisfied in this case (see \cite[ Theorem 3.1]{BaudoinBonnefontGarofalo2014}). We refer to
\cite[Sec. 3, pp.3184-3186]{Baudoin-Gordina-Mariano-2019} for the details of the construction of the Dirichlet form related to the operator $L$.  Therefore we conclude that Theorem \ref{thm:Prob-Estimate}, Corollary  \ref{Cor:polynomial-bound}, Proposition \ref{Cor:main_product_bound}, the upper bound  \eqref{eq:Asymp1'} of Proposition \ref{thm:Main_Asymptotic} and Proposition \ref{MainThm2} also hold. Using the same argument as in the Carnot group case, the asymptotic \eqref{eq:Asymp1} also holds for bounded and non-characteristic domains.

\subsection{Sierpinski gasket and fractal-like manifolds with varying walk-dimensions}\label{sec:fractals}

The existence of ``Brownian motion" on fractals and its heat kernel has been extensively studied for decades (see \cite{Barlow-Perkins-1988,Goldstein-1987,Kusuoka-1987} for example). 
The Sierpinski gasket $K$ can be defined as the unique compact subset in $\C$ such that 
\[
K=\cup_{i=1}^3 f_i(K)
\]
where $f_i(z)=\frac{z-p_i}{2}+p_i$, $i=1,2,3$ for all $z\in \C$ and $V:=\{p_1,p_2, p_3\}$ are the vertices of an equilateral triangle of side length $1$ in $\mathbb{C}$. One can define the unbounded Sierpinski gasket as 
\[
\hat{K}=\cup_{n=0}^\infty 2^n K.
\]
To construct $\hat{K}$ one can define,
\[
V_0=\cup_{m=0}^\infty 2^m\left( \cup_{i_1,\dots, i_m=1}^3 f_{i_1}\circ \cdots \circ f_{i_m}(V) \right), \quad V_m=2^{-m}V_0,
\]
so that $\hat{K}$ is given by the closure of $\cup_{m\ge0}V_m$.
 It is known that the Hausdorff dimension of $K$ with respect to the Euclidean metric (denoted by $d(x,y)=|x-y|$) is $\frac{\ln3}{\ln2}$. A Hausdorff measure on $K$ can be given by the Borel measure $\mu$ on $K$ such that for all $m\ge1$, $i_1,\dots, i_m\in \{1,2,3\}$,
\[
\mu(f_{i_1}\circ \cdots \circ f_{i_m}(K))=\frac{1}{3^m}.
\]
The metric $d$ and measure $\mu$ can be extended on $\hat{K}$ correspondingly. Note that the measure $\mu$ is $\frac{\ln3}{\ln2}$-Ahlfors regular, thus the metric measure space $(\hat{K}, d, \mu)$ satisfies $\left(\VD\right)$ and $\left(\RVD\right)$ with $\alpha=\alpha'=\frac{\ln3}{\ln2}$. 
For any function $f:V_m\to \R$, the discrete Laplacian on $V_m$ is defined by
\[
\Delta_m f(p)=5^m \sum_{p'\in U_{m,p}}(f(p')-f(p)),\quad p\in V_m\setminus V_0,
\] 
where $U_{m,p}$ denotes the collection of neighbors of $p$ in $V_m$. {     Let $\mathcal{D}$ be the set of continuous functions on $K$ such that $\lim_{m\to \infty}\Delta_m f(p)=g(p)$ uniformly for $p\in\cup_{m\ge 0}V_m\setminus V_0$ for some $g\in C(K)$. }
The Laplacian $\Delta$ on the Sierpinski gasket $\hat{K}$ can be given by 
\begin{equation}\label{eq:Laplacian}
\Delta f(p)=\lim_{m\to \infty}\Delta_m f(p), \quad f\in \mathcal{D},\quad p\in \cup_{m\ge0}V_m\setminus V_0.
\end{equation}
As in \cite{Barlow-Perkins-1988}, one can construct a Brownian motion on $\hat{K}$ and the corresponding heat kernel satisfies $\left(\ULE\right)$ with $\beta=\frac{\log5}{\log2}>2$,  
\[
p_{\hat{K}}\left(x,y,t\right)\asymp\frac{C}{t^{\alpha/\beta}}\exp\left(-c\frac{d\left(x,y\right)^{\beta/\left(\beta-1\right)}}{t^{1/\left(\beta-1\right)}}\right).
\]

Other examples include affine nested fractals, post critically finite self-similar fractals, among many others. We refer the reader to references such as \cite{Barlow-1998, Fitzsimmons-Hambly-Kumagai-1994,Hambly-Kumagai-1999,Kigami-2009} for heat kernel sub-Gaussian estimates and the verification of the assumptions of our theorems on a large classes of Fractals.

Another set of examples that is worth mentioning here are the fractal-like  manifolds, whose walk dimension $\beta$ varies upon the consideration of short time or long time. For instance, consider the $2$-dimensional Riemannian manifold  $M$ that is made from the Sierpinski gasket graph $V_0$ by 1) replacing the edges by tubes of length $1$; 2) gluing the tubes together smoothly at the vertices  allowing some small bumps and  removing some of the tubes. It is known that $M$ is roughly isometric (in the sense of \cite{Kan1, Kan2}) to the Sierpinski gasket graph $V_0$ and satisfies $\left(\UEF\right)$ with (see \cite{Barlow-Bass-Kumagai-2006})
\[
F(r)=\begin{cases}
r^{2} & r<1\\
r^{\log5/\log2} & r\geq1
\end{cases}.
\]
Namely there exists constants $\Cue,\cue>0$ such that  the heat kernel satisfies 
\begin{equation*}
p_{M}(x,y,t)\leq\frac{\Cue}{V\left(x,t^{1/2}\right)}\exp\left(-\cue\left(\frac{d\left(x,y\right)^{2}}{t}\right)\right),\quad  t\leq1\vee d(x,y)
\end{equation*}
and
\begin{equation*}
p_{M}(x,y,t)\leq\frac{\Cue}{V\left(x,t^{\log2/\log5}\right)}\exp\left(-\cue\left(\frac{d\left(x,y\right)^{\log5/\log2}}{t}\right)^{\frac{1}{\log5/\log2-1}}\right), \quad t\geq1\vee d(x,y).
\end{equation*}

In fact the fractal-like manifold $M$ satisfies $\left(\ULE\right)$ with $\beta=2$ in $\left\{ \left(t,x,y\right):t\leq1\vee d(x,y)\right\} $ and $\beta=\log5/\log2$ in $\left\{ \left(t,x,y\right):t\geq1\vee d(x,y)\right\}$.


\begin{acknowledgement}
We thank Rodrigo Ba\~nuelos, Fabrice Baudoin,  Li Chen, and Maria Gordina for helpful discussions on improving the paper. We thank Tuomas Hyt\"onen for a remark that helped simplify one of our arguments. We also thank two anonymous referees whose insightful comments and suggestions have led to an improved manuscript.
\end{acknowledgement}

\bibliographystyle{plain}	
\bibliography{MMS-MainRef}

\end{document}